\documentclass[a4paper, final, onecolumn, twoside, 11pt, runningheads]{svjour3_HAL}  

\smartqed  
\usepackage{graphicx}
\usepackage[english]{babel}
\usepackage{latexsym}
\usepackage[T1]{fontenc}
\usepackage{type1cm} 
\usepackage{times}
\usepackage[ansinew]{inputenc}
\usepackage{graphicx,subfigure}
\usepackage{amsmath,amssymb,amscd,amsfonts}
\usepackage{array}
\usepackage{bbold}
\usepackage{xcolor}
\usepackage[numbers]{natbib} 
\usepackage{url}
\usepackage[dvipdfm]{hyperref} 
\hypersetup{ 
menucolor=red,
anchorcolor=black,
urlcolor=blue,
linkcolor=blue,
filecolor=magenta,
citecolor=red,
colorlinks=true,
breaklinks=true,
pdftitle={Transportation Distances on the Circle and Applications},
pdfauthor={Julien Rabin, Julie Delon and Yann Gousseau},
pdfsubject={JMIV special issue},
   pdfpagemode=UseOutlines, 
   pdfstartview=FitV,
   pdffitwindow=true,
pdfkeywords={}
plainpages=false
}

\newenvironment{proofof}[1]{{ \em \bf Proof of #1}}{$\hfill \Box$ \smallskip} 
\newcommand \eps {\varepsilon}

\def\R{{\mathbb{R}}}

\def\eg{{\em {e.g.}~}}
\def\etc{{\em {etc}}}
\def\etal{{\em {et al.}~}}
\def\eps{{\varepsilon}}

 \journalname{HAL}
\graphicspath{ 
  {figures/}
}
\DeclareGraphicsExtensions{.eps,.ps}

\begin{document}

\title{Transportation Distances on the Circle and Applications
\thanks{This work has been realized under grant BLAN07-2\_183172.}
}

\author{Julie~Delon \and  Julien~Rabin \and Yann~Gousseau}

\titlerunning{Transportation Distances on the Circle and Applications} 
\authorrunning{Delon, Rabin and Gousseau} 

\institute{ Julie Delon and Yann Gousseau \at
           CNRS LTCI T\'{e}l\'{e}com ParisTech, 46 rue Barrault, 75634 Paris Cedex 13 \\
            Tel.: +33-145-817-073\\
            Fax: +33-145-813-794\\
            \email{julie.delon@telecom-paristech.fr and yann.gousseau@telecom-paristech.fr} 
           \and
           Julien Rabin \at
           CNRS C\'{E}R\'{E}MADE Université Dauphine, Place du Maréchal De Lattre De Tassigny, 75775 Paris Cedex 16 \\
            Tel.: +33-144-054-923\\
            Fax: +33-144-054-599\\
            \email{julien.rabin@ceremade.dauphine.fr}  
}

\date{Received: date / Accepted: date}

\maketitle

\begin{abstract}
This paper is devoted to the study of the Monge-Kantorovich theory of optimal mass transport and its applications, in the special case of \emph{one-dimensional and circular distributions}.
More precisely, we study the Monge-Kantorovich distances
between discrete sets of points on the unit circle $S^1$, in the case where
the ground distance between two points $x$ and $y$ is defined as
$h(d(x,y))$, where $d$ is the geodesic distance on the circle and $h$ a
convex and increasing function.
 We first prove that computing a Monge-Kantorovich distance between two given sets of pairwise different points boils down to cut the circle at a well chosen point and to compute the same distance on the real line. 
This result is then used to obtain a metric between 1D and circular discrete histograms, which can be computed in linear time. 
A particular case of this formula has already been used
in~\cite{ref:rabin09siam} for the matching of local features between images, involving circular histograms of gradient orientations.
In this paper, other applications are investigated, in particular dealing with
the hue component of color images.
In a last part, a study is conducted to compare the advantages and drawbacks
of transportation distances relying on convex or concave cost functions, and of the classical $L^1$ distance.

 \PACS{PACS code1 \and PACS code2 \and more}
 \subclass{MSC code1 \and MSC code2 \and more}

\end{abstract}

\section{Introduction}\label{sec:intro}

The theory of optimal transportation was first introduced by 
Monge~\cite{monge} in its \emph{Mémoire sur la théorie des déblais et 
des remblais} (1781) and rediscovered by Kantorovich~\cite{Kantorovich} 
in the late '30s.  The Monge-Kantorovich problem can be described in the 
following way. Given two probability distributions $f$ and $g$ on $X$ and
$c$ a nonnegative measurable cost function on $X \times X$, the aim is to find the optimal transportation cost
\begin{equation}\label{eq:def_MongeKantorovich}
\textsc{MK}_c (f,g):= \inf_{\pi \in \Pi(f,g)} \iint_{x,y} c(x,y)\, 
d\pi(x,y)  \, \;
\end{equation}
where $\Pi(f,g)$ is the set of probability measures on  $X \times X$ 
with marginals $f$ and $g$ (such measures are called transportation 
plans). The existence, uniqueness and behavior of optimal transportation 
plans has been thoroughly studied in the last 
decades\cite{ref:villani03book,Brenier_book,mccann1999est,villani2008optimal,mccann2,mccann3}. 

This framework is nowadays widely used in many fields of research, such as cosmology~\cite{reconstruction_universe}, meteorology~\cite{Cullen_book}, fluid mechanics or electromagnetic (see~\cite{Brenier_book} for a complete review).

The use of the Monge-Kantorovich framework in image processing and computer vision
has been popularized by Rubner \etal~\cite{rubner_ijcv} for image retrieval
and texture classification with the introduction of the so-called Earth
Mover's Distance (EMD). Although the definition of the EMD is
slightly different from the original Monge-Kantorovich formulation, these are
equivalent when considering distributions having the same total weight. In the following years and up to now,
a large body of works has relied on the use of such distances for image
retrieval, see e.g.~\cite{EMD_image_retrieval_blobworld,Lv_retrieval,EMD_image_retrieval_semantic,EMD_image_retrieval_semantic2,Hurtut08,Pele_iccv2009}. This extensive use of transportation distances is largely
due to their robustness when comparing histograms or discrete
distances. For the same reason, these distances are also successfully used to compare local
features between images,
see~\cite{ling_EMD,pele2008eccv,ref:rabin08icprcemd,ref:rabin09siam,Pele_iccv2009}. Other uses of transportation
distances for images include: image registration~\cite{Tannenbaum04}, image
morphing~\cite{Tannenbaum07} or junction detection~\cite{tomasi_PAMI01}. 

The strongest limitation of transportation distances is their computational
cost. Standard approaches quickly become intractable when dealing with a large amount of data in dimensions
more than two. Indeed, the simplex algorithm, interior point methods or the
Hungarian algorithm all have  a complexity of at least $0(N^3)$ ($N$ being the
size of the data, either the number of samples or the number of histogram bins). Therefore, several works have proposed to speed up the
computation or the approximation of optimal transport, in particular in the
field of image processing, where the amount of data is often massive,
see~\cite{image_retrieval_EMD_embedding,fast_contour_matching_EMD_embedding,ling_EMD,EMD_approx_wavelet}. One particular case in which the computation is elementary and fast is the case of
\emph{one-dimensional histograms}, for which it is well known that optimal transport,
in the case of a \emph{convex cost function}, reduces to the pointwise difference between
cumulative distribution functions~\cite{ref:villani03book}. A question that arises is
then the possibility to perform such simple and efficient computations in
the case of {\it circular} histograms, i.e. histograms in which the first and
last bins are neighbors.

Indeed, circular histograms are especially important in image processing and
computer vision. First, the local geometry is often efficiently coded by the distribution
of gradient orientations. Such representations offer the advantage of being
robust to various perturbations, including noise and illumination
changes. This is in particular the case for the well known SIFT~\cite{SIFT}
descriptor and its numerous variants. In such a situation, the comparison of local
features reduces to the comparison of one-dimensional \emph{circular}
histograms. Other local features involving circular histograms include the
so-called Shape Context~\cite{shape_context}. Second, the color content of an image is
often efficiently accounted for by its \emph{hue}, in color spaces such as HSV or
LCH. In such cases again, information is coded in the form of circular
one-dimensional histograms. Several works in the field of computer vision
have explicitly addressed the use of transportation distances in the case of
circular histograms, either using thresholded concave cost
functions~\cite{pele2008eccv,Pele_iccv2009} or $L^1$ cost
functions~\cite{ref:rabin08icprcemd,ref:rabin09siam}. 

The goal of this paper is first to give a general formulation of
transportation distances when the cost function is a convex function of the
Euclidean distance on the circle. This formulation gives a practical way to
compute distances in linear time in this case. Second, we provide various
experiments of image manipulation or retrieval for which the interest of
circular transportation distances is shown. Eventually, we conclude with a
discussion (that actually applies to both circular and non-circular cases) on the
respective interest of transportation distances with either convex or concave
cost functions when compared to classical bin-to-bin distances. It is shown
that the choice between these three family of distances should essentially be
driven by  the type of perturbation the histograms are likely to suffer from. 

\paragraph{\bf \em Outline}\label{sec:outline}

The paper is organized as follows. In
Section~\ref{sec:transportation_on_circle}, the optimal transportation flow of
the Monge-Kantorovich problem is investigated in the circular case. The
definition of this problem being recalled, a new formula is introduced and a
sketch of the proof is proposed (details of the proof are provided in
appendix~\ref{sec:annexe}). In section~\ref{sec:applications}, several
applications are studied to show the interest of such a metric for image processing and computer vision.
First, an application to hue transfer between images is proposed in
\S~\ref{sec:hue_transfer} as a result of an optimal transportation flow on the circle.
Then, in \S~\ref{sec:hist_comparison}, some applications of histogram
comparison for image retrieval are proposed. Eventually, the discussion
about the robustness and the limitations of Monge-Kantorovich distances in the
framework of histogram comparison is given in Section~\ref{sec:comparative_study}.

\section{The Monge-Kantorovich transportation problem on the circle}\label{sec:transportation_on_circle}

In this section, we present some results on the Monge-Kantorovich distances between two circular histograms. In particular, we give an analytic formulation of these distances when the ground cost between points on the circle can be written as an increasing and  convex function of the Euclidean distance along the circle.

\subsection{Definitions}\label{sec:definitions}

Consider two discrete and positive distributions 
\begin{equation}
  \label{eq:discrete_histograms}
  f=\sum_{i=1}^{N} f[i]\delta_{x_i} \;\;\text{ and }\; \; g=\sum_{i=1}^{M} g[j]\delta_{y_j},
\end{equation}
where $\{x_1,\dots x_N\}$ and $\{y_1,\dots y_M\}$ are 
two sets of points on a subset $\Omega$ of $\R^K$.
Assume that these distributions are normalized in the sense that $\sum_{i=1}^{N} f[i]=\sum_{j=}^{M} g[j]=1$. Let $c : \Omega \times \Omega \mapsto \R^+$ be a nonnegative cost function (called \textit{ground cost}), the quantity 
\begin{equation}
  \label{eq:monge-kantorovich_c}
  \textsc{MK}_{c}\,(f,g):=\min_{(\alpha_{i,j}) \in \mathcal{M}}  \sum_{i=1}^N \sum_{j=1}^M \alpha_{i,j} c({\textstyle x_i, y_j}), \text{ with }\;\;
\end{equation}
\begin{equation}
\label{eq:matrix}
\mathcal{M}=\{(\alpha_{i,j}) \in \R^N \times \R^M;\; \alpha_{i,j} \geq 0,\; \sum_i \alpha_{i,j}= g[j],\; \sum_j \alpha_{i,j}=f[i]\},
\end{equation}
is called the \textit{optimal transportation cost} between $f$ and $g$ for the ground cost $c$.
Matrices $(\alpha_{i,j})$ in $\mathcal{M}$ are called\textit{ transport plans} between $f$ and $g$. If  $(\alpha_{i,j})$ is optimal for~\eqref{eq:monge-kantorovich_c}, we say that $(\alpha_{i,j})$ is an \textit{optimal transport plan}. 

Let $d$ be a distance on  $\Omega$ and assume that the ground cost can be written $c(x,y) = d(x,y)^{\lambda}$, with the convention $d(x,y)^0 = \mathbf{1}_{x \neq y}$.  It can be shown~\cite{ref:villani03book} that
\begin{itemize}
\renewcommand{\labelitemi}{$\bullet$}
\item[$\bullet$] when  $\lambda \in [0,1[ $,  $\textsc{MK}_c$ is a distance between probability distributions ;
\item[$\bullet$] when $\lambda \in [1,\infty[$,
  $ \textsc{MK}_{\lambda}\,(f,g):=\left(\textsc{MK}_c(f,g)\right)^{\frac{1}{\lambda}} $
is also a distance between probability distributions.
\end{itemize}
These distances are called \textit{Monge-Kantorovich distances}, or Wasserstein distances. For $\lambda=1$, $\textsc{MK}_1$ is also known as the Kantorovich-Rubinstein distance, or in computer vision as the \textit{Earth Mover's Distance} (EMD), as introduced by Rubner in~\cite{rubner_ijcv}. 

Computing optimal transportation costs is generally not an easy task. The main exception is the case of the real line: if $\Omega = \R$, and if the cost $c$ is  a convex and increasing function of the euclidean distance $|x-y|$, then the optimal transport plan between $f$ and $g$ is the monotone rearrangement of $f$ onto $g$, which sends the mass starting from the left. This result is usually false if $c$ is not a convex function of the euclidean distance on the line.

In the following, we take interest in the case where $\Omega$ is a circle $S^1$ of perimeter 1, and where $c$ is an increasing function of the geodesic distance $d$  along the circle.
In particular, we will see that the previous result on the line can be generalized in this case.

\subsection{Optimal transportation for convex functions of the distance }\label{sec:computing_Monge_circle}
The main result of this section is an analytic formulation of the optimal transportation cost between the discrete distributions $f$ and $g$ on the circle $S^1$ when the ground cost $c$ can be written $c(x,y) = h(d(x,y))$, with $h:\R \to \R^+$ an increasing and convex function and $d$ the geodesic distance along the circle.  In the following, we use the same notations for points on the circle and their coordinates along the circle, regarded as variables taking their values on the reduced interval $[0,1[\;(modulo\,1)$. It follows that $d$ can be written
\begin{equation}
  \label{eq:distance}
d(x,y)=\min(|x-y|, 1-|x-y|).
\end{equation}
The distributions $f$ and $g$ on $S^1$ can be seen equivalently as periodic distributions of period $1$ on $\mathbb{R}$. 

Let us define the cumulative distribution function of $f$  on $[0,1[$ as
\begin{equation}
  \label{eq:cumulative}
 \forall y \in [0,1[,\;\;  F(y)=  \sum_{i=1}^N  f[i] \cdot \mathbb{1}_{\{ x_i \in [0,y[ \}} \;.
\end{equation}
$F$ is increasing and left continuous, and can be extended on the whole real line with the convention $F(y+1) = F(y) +1$. This boils down to consider $f$ as a periodic distribution on $\R$. We define also the pseudo-inverse  of $F$ as $F^{-1}(y)=\inf \{t;\; F(t)>y\}.$ The interest of these definitions lies in the next result.

\subsubsection{An analytic formulation of optimal transportation on the circle}

\begin{theorem}
  \label{theoremMK}
\textit{
Assume that $d$ is given by Equation~\eqref{eq:distance} and that the ground cost $c$ can be written $c(x,y) = h(d(x,y))$, with $h:\R \to \R^+$ an increasing and convex function. Let $f$ and $g$ be two discrete probability distributions on the circle, with cumulative distribution functions $F$ and $G$, and let  $G^{\alpha}$ denote the function $G - \alpha$. Then
\begin{equation}
  \label{eq:MKresult}
  \textsc{MK}_{c}\,(f,g)=\inf_{\alpha \in \R} \int_0^1 h( |F^{-1} - (G^{\alpha})^{-1} |) \;.
\end{equation}
}
\end{theorem}

\paragraph{\bf \em Idea of the proof}
This result is a generalization of the real line case, where it is well known~\cite{ref:villani03book} that the global transportation cost between two probability distributions $f$ and $g$ can be written 
\begin{equation}
  \label{eq:MKresult_real}
  \textsc{MK}_{c}\,(f,g)=\int_0^1 h ( | F^{-1}-G^{-1} |) \;.
\end{equation}
A proof of Equation~\eqref{eq:MKresult} in a continuous setting has been proposed very recently in ~\cite{delon09siap}, where it is shown that this equation holds for any couple of probability distributions. However, this proof involves some complex notions of measure theory which are not needed in the discrete setting.
For the sake of completeness and simplicity, we provide in Appendix~\ref{sec:annexe} a simpler proof of these theorem in the case of discrete distributions. The proof first focus on the case where $f$ and $g$ can be written as sums $f = \frac{1}{P} \sum_{k=1}^P \delta_{x_k}$, and $g = \frac{1}{P} \sum_{k=1}^P \delta_{y_k},$ where $\{x_1,\dots x_P\}$ and $\{y_1,\dots y_P\}$ are discrete sets of points on the unit circle $S¹$. When the points are all pairwise different, we show that the circle can always be ``cut'' at some point, such that computing the optimal transport between $f$ and $g$ boils down to compute an optimal transport between two distributions on the real line (see Figure~\ref{fig:cut_circle}).
This result is proven first for strictly convex functions $h$ and for any optimal transport plan, then for any convex function $h$ and a well chosen optimal plan. Once the problem has been reduced to the real line, Formula~\eqref{eq:MKresult} follows from the fact that the optimal transport on $\R$ is given by the ordering of the points. 
The generalization of this formula to any kind of discrete distribution results from the continuity of the global transport cost $\textsc{MK}_c\,(f,g)$ in the values of the masses and their positions on the circle.

\begin{figure*}[htb]
   \centering
   \includegraphics[height=0.2\textheight]{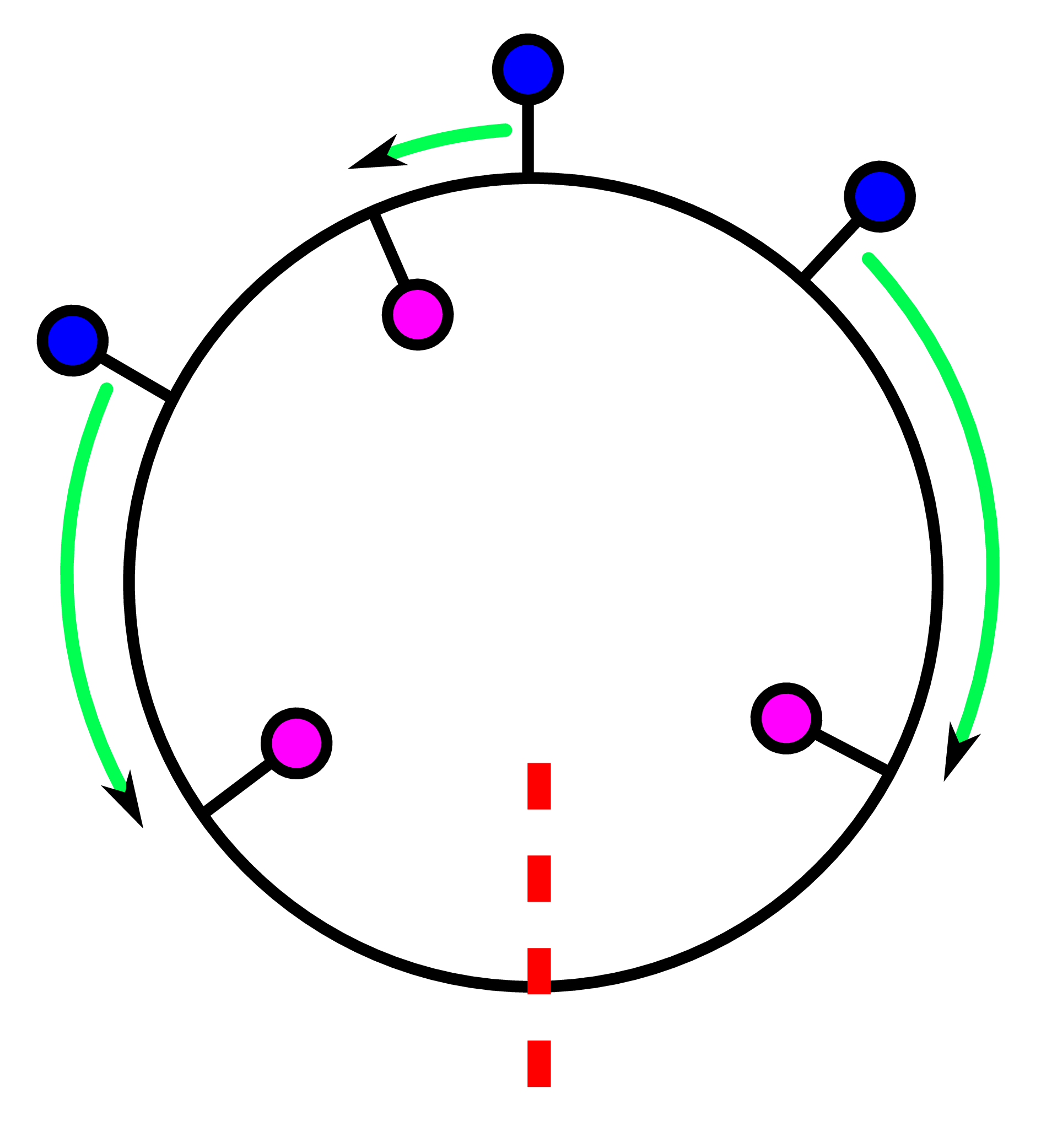} 
   \hspace*{0.05\textwidth}
   \includegraphics[height=0.15\textheight]{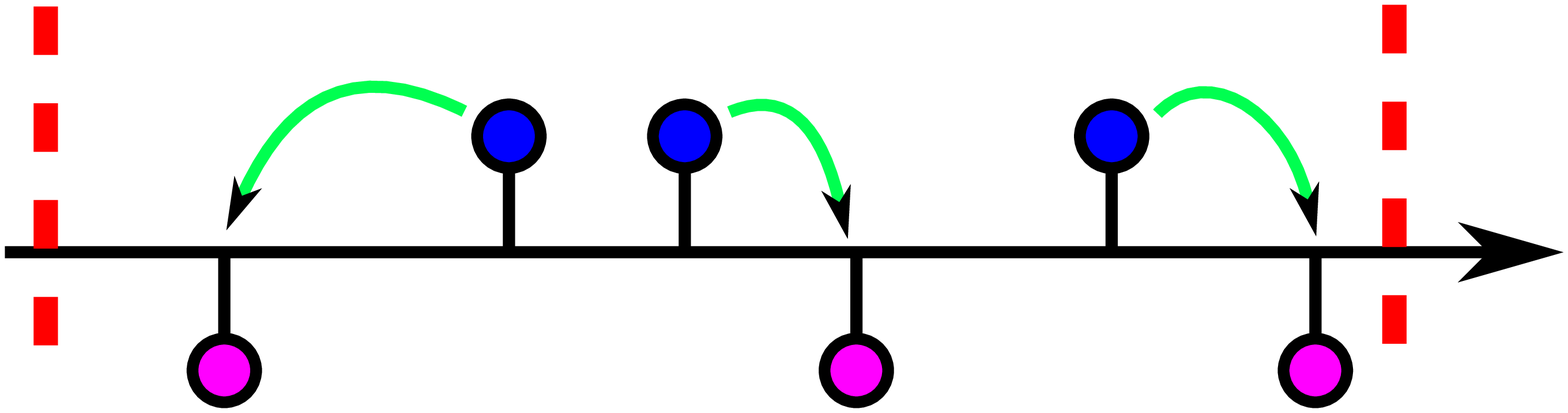}
   \caption{When the distributions are sums of unitary different masses, and when the ground cost is a nonnegative, increasing and convex function of the distance along the circle, there is a (non-necessarily unique) ``cut'' on the circle such that the optimal transportation  on $S^1$ boil down to optimal transportation on the real line.
   }
   \label{fig:cut_circle}
\end{figure*}

In practice, Formula~\eqref{eq:MKresult} can be computed for any $h$  at a precision $\eps$ with a complexity in $O((N+M)\log \frac 1\eps)$~\cite{delon09siap}, where $N$ and $M$ are the number of points in the distributions $f$ and $g$ ({\em i.e.}  the number of masses in the distributions). 

\subsubsection{The case $c(x,y) = d(x,y)$}\label{sec:cemd}

If $h$ is a power function $x \mapsto x^{\lambda}$, with $\lambda \geq 1$, Theorem~\ref{theoremMK} gives us a way to compute Monge-Kantorovich distances between $f$ and $g$:
\begin{equation}
   \label{eq:MKlambdaresult}
  \textsc{MK}_{\lambda}\,(f,g)=\left(\inf_{\alpha \in \R} \int_0^1|F^{-1} - (G^{\alpha})^{-1} |^{\lambda}\right)^{\frac 1 \lambda}.
\end{equation}

In the case $\lambda=1$ ({\em i.e.} when the ground cost $c$ is the distance $d$ along the circle), this result can be rewritten
\begin{equation}
  \label{eq:MKresult1}
  \textsc{MK}_1\,(f,g)=\inf_{\alpha \in \R} \int_0^1|F-G-{\alpha}|.
\end{equation}

Observe that an alternative proof of Equation~\eqref{eq:MKresult1} was proposed by Werman et al. in~\cite{werman1986bgm} for distributions written as sums of unitary masses. A similar result is shown in~\cite{cabrelli1995kantorovich} for the Kantorovich-Rubinstein problem, which is known to be equivalent (see~\cite{ref:villani03book}, chapter 1) to the Monge-Kantorovich problem when the cost $c(x,y)$ is a distance, which is true for $\lambda =1$ (but false for $\lambda >1$).

In practice, since $F-G$ is piecewise constant for discrete distributions, the infimum of Equation~\eqref{eq:MKresult1}  can be computed easily  by computing the weighted median of the (finite number of) values $F(t)-G(t)$ when $t \in [0,1[$, the weights being the lengths of the intervals on which $F-G$ is constant. In practice, this yields a $O(N)$ exact algorithm to compare two normalized distributions of $N$ masses~\cite{cabrelli1998linear,gurwitz1990weighted}. 

\subsubsection{Discrete histograms}

Most applications deal with discrete histograms, {\em i.e.} discrete distributions living on a uniform grid of $N$ bins. In the case of histograms on the real line, for the cost $c(i,j) = |i-j|$, Formula~\eqref{eq:MKresult_real} becomes
\begin{equation}
  \label{eq:emd_bin}
  \text{EMD}\,(f,g) = \textsc{MK}_1(f,g) =   \|F-G\|_1 ,
\end{equation}
where we denote by $\|.\|_1$ the discrete $L^1$ norm, by $F$ and $G$  the cumulative histograms of $f$ and $g$, and where $\text{EMD}$ is the acronym for Earth Mover's Distance~\cite{rubner_ijcv}. An illustration is given in Figure \ref{fig:hist_cumul}.

In the case of circular histograms, bins $0$ and $N-1$ are neighbors. If the cost $c$ is $c(i,j) = \min(|i-j|, N - |i-j|)$ along the circle, 
 Formula~\eqref{eq:MKresult1} can be rewritten
\begin{equation}
  \label{eq:cemd_med}
\textsc{cemd}\,(f,g) := \textsc{MK}_1\,(f,g)=\inf_{\alpha} \sum_{i=0}^{N-1} | F[i]-G[i]-{\alpha}|
  = ||F-G-{\mu}||_1 \,,
\end{equation}
 where $\mu$ is the median of the set of values $\{ F[i]-G[i], 0\leq i\leq N-1
 \}$, and where $\textsc{cemd}$ is the acronym for Circular Earth Mover's
 Distance, as defined in~\cite{ref:rabin09siam}. Indeed, it is easily checked
 that the distance defined by Formula~(\ref{eq:cemd_med}) is equivalent to the distance introduced
 in~\cite{ref:rabin09siam}, that is 
\begin{equation}\label{eq:cemd_bin}
   \textsc{cemd}\,(f,g) =  \min_{k \in \{0,\dots, N-1\}} \|F_k-G_k\|_1,
\end{equation}
 where ${F}_k[i]$ is defined as $F[i]-F[k]$ if $i \in \{k,\dots N-1\}$ and $F[i]-F[k]+1$ if $i \in \{0,\dots k-1\}$ 
  (the definition being similar for ${G}_{k}$ by replacing ${f}$ by
  $g$). 
In other words, the distance $\text{\sc MK}_1(f,g)$ is the minimum in $k$ of
the $L^1$ distance between ${F}_k$ and ${G}_k$, the cumulative histograms of
${f}$ and ${g}$ starting at the $k^{th}$ quantization bin. 

\begin{figure*}[htb]
   \centering
   \includegraphics[height=0.2\textheight]{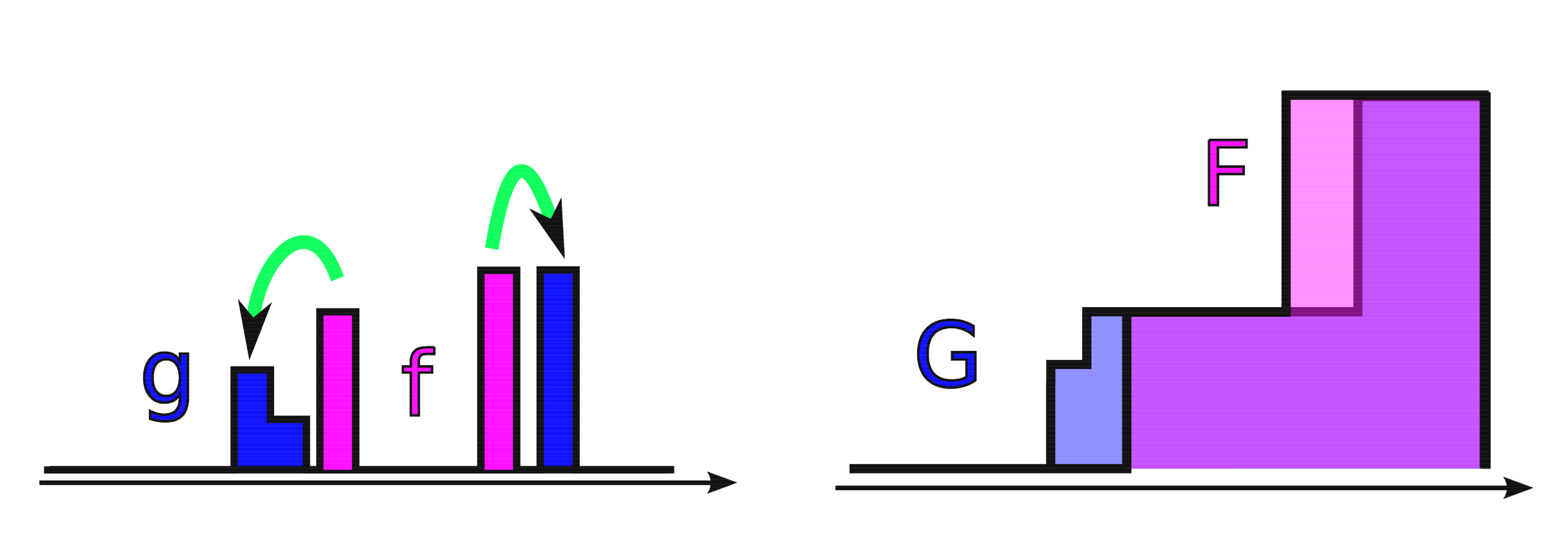}
   \caption{Using $c(x,y) = |x-y|$ on the real line, the optimal transportation plan between two discrete histograms $f$ and $g$ is the $L^1$ distance of the difference between the cumulative histograms ${F}$ and ${G}$ (Formula~\eqref{eq:emd_bin}).}
   \label{fig:hist_cumul}
\end{figure*}

\subsection{Optimal transportation for concave functions of the distance}

In practice, it may be useful to choose the  ground cost $c$ as a nonnegative, concave and increasing function $h$ of the ground distance $d$. For instance, for the task of image retrieval, several authors~\cite{rubner_ijcv,Hurtut08, tomasi_PAMI01}) claim that optimal results can be achieved with a function 
\begin{equation}
  \label{eq:function_exponentielle}
h(t) = 1-e^{-\frac{t}{\tau}}.
\end{equation}
Notice that if $h$ is increasing, concave and such that $h(0) = 0$, it is easy to show that $h(d)$ is also a distance, and thus $\textsc{MK}_c$ is also a distance between probability distributions. Another property  of concave costs is that they do not move the mass which is shared by the distributions~\cite{ref:villani03book}: if $f$ and $g$ are histograms, the problem is reduced to the comparison of $(f-g). \mathbf{1}_{f -g \geq 0}$ with  $(g-f). \mathbf{1}_{f -g < 0}$, which have disjoint supports.

However, in the case of such concave functions $h$,  Theorem~\ref{theoremMK} does not apply, and there is no general and fast algorithm to compute corresponding optimal transportation costs, either on the real line or on the circle. In most cases, we are reduced to use linear programming, {\em i.e.} simplex or interior point algorithms, which are known to have at best a $O(N^3 \log N)$ complexity to compare two histograms on $N$ bins~\cite{burkard-assignements}. 
We describe in the following some special cases of concave function $h$ for which this complexity can be reduced.

\subsubsection{$L^1$ as a Monge-Kantorovich distance}\label{sec:L1_Monge}

If the distributions $f$ and $g$ are discrete histograms on $N$ bins,  and if $h(t) = \mathbf{1}_{t \neq 0}$, then the Monge-Kantorovich distance between $f$ and $g$ is~\cite{ref:villani03book}
\begin{equation}
  \label{eq:MK_indicatrice}
  \textsc{MK}_{\mathbf{1}_{d(x,y) \neq 0}}(f,g) = \frac{1}{2} \sum_{i=1}^N |f[i]- g [i]| = \|f-g\|_1.
\end{equation}
In other words, the $L^1$ distance between two normalized histograms $f$ and $g$ is a Monge-Kantorovich distance for the concave function $\mathbf{1}_{t \neq 0}$.

\subsubsection{Thresholded distances}
\label{sec:thresholded}
In~\cite{pele2008eccv,Pele_iccv2009}, Pele and Werman consider thresholded ground distances, using $h(t) = \min(t,T)$, with $T$ a given threshold. Up to a multiplicative factor, this function $h$ can almost be seen as a discrete version of~\eqref{eq:function_exponentielle}, where $\tau$ is chosen proportional to $T$. They show that in this case, the computation of the optimal cost can be solved by a ``min-cost-flow algorithm'', whose complexity is smaller than classical linear programming algorithms. In their experiments, they use $T=2$ for comparing histograms on $N$ bins, which means that the cost $c(i,j)$ can take only three values: $0$ if $i=j$, $1$ if $i$ and $j$ are neighbors, and $2$ in other cases. Since all the shared mass remains in place, we know that if $(\alpha_{i,j})$ is the optimal transport plan between $f$ and $g$, then for a given $i$, $\sum_{j\neq i}\alpha_{i,j}  = (f[i] - g[i]) \mathbf{1}_{f[i] \geq g[i]}$, which implies in particular that
\begin{eqnarray*}
  \label{eq:thresholded_distances}
\textsc{MK}_{\min(d,2)}(f,g) &=&  \sum_{i=1}^N  ((\sum_{j \neq i } 2 \alpha_{i,j}) - \alpha_{i,i+1} - \alpha_{i,i-1})\\
&=& 2 \sum_{i=1}^N (f[i] - g[i]) \mathbf{1}_{f[i] \geq g[i]} - \sum_{i=1}^N (\alpha_{i,i+1} + \alpha_{i,i-1}) \\
&=& \|f-g\|_1 - \sum_{i=1}^N (\alpha_{i,i+1} + \alpha_{i,i-1}).
\end{eqnarray*} Now, notice that $\alpha_{i,i+1} $ is different from $0$ only
if $f[i] \geq g[i]$ (otherwise, all the mass in $i$ stay in place) and $f[i+1]
< g[i+1]$ (otherwise the mass $g[i+1]$ is already ``filled'' by a part of
$f[i+1]$). In other words, the only points where the quantities
$\alpha_{i,i-1}$ or $\alpha_{i,i+1}$ are different from $0$ are the points
where the densities of $f$ and $g$ are crossing. It follows that in many
cases, the thresholded distance $\textsc{MK}_{\min(d,2)}$ is very close to
$L^1$, in particular when $N$ is large and when histograms are crossing at
only a few places, as we will see in the experiments of
Section~\ref{sec:comparative_study}. In order to allow larger ground
displacements, the use of values of $T$ larger than 2 is proposed
in~\cite{Pele_iccv2009}. This is made at the price of a non-linear complexity
and necessitates a compromise in the tuning of $T$ (smaller values yield
faster computations). We will come back on the use of such concave cost
functions in Section~\ref{sec:comparative_study}.

In the following section, we illustrate the interest of circular
Monge-Kantorovich distances for several applications.

\section{Experimental study}\label{sec:applications}
This section is devoted to an experimental analysis of the previous optimal transportation framework for different computer vision tasks.
We first consider the color transfer from one image to another as an application of the optimal transportation flow on the circle between hue histograms.
In the following, two different applications of transportation distances for the comparison of histograms are studied: local feature comparison and image retrieval.

\subsection{Hue transfer between color images}\label{sec:hue_transfer}
The aim of this paragraph is to use the optimal transportation framework to transfer a hue distribution from one image to another.

\paragraph{\bf \em Contrast transfer}
First, let us recall that histogram equalization and more generally histogram specifications are  merely particular cases of optimal transportation on the real line. Indeed, if $u:\Omega \mapsto \{0,\dots, L-1\}$ is a discrete image and $h_u$ its gray level distribution, histogram specification consists in finding the optimal transport plan between $h_u$ and a target discrete probability distribution $h_t$ on $ \{0,\dots, L-1\}$ (one speaks of histogram equalization when $h_t$ is  a constant distribution on $ \{0,\dots, L-1\}$). If one considers a cost $c$ equal to the euclidean distance on the line, then, as explained in Section~\ref{sec:definitions}, the solution of this problem consists in a monotone rearrangement. This rearrangement is obtained by applying the function $H_t^{-1} \circ H_u$ to $u$, where $H_u$ (resp. $H_t$) is the cumulative distribution function of $h_u$ (resp. $h_t$) and $H_t^{-1}$ represent the pseudo-inverse of $H_t$ (see definition in section~\ref{sec:computing_Monge_circle}, after Formula~\eqref{eq:cumulative}).
 If $u$ is a  color image, such contrast adjustments can be applied to its ``intensity'' channel (\eg the channel ``Value'' in the HSV representation).
 
\paragraph{\bf \em Hue transfer}
Thanks to Formula~\eqref{eq:cemd_med} or ~\eqref{eq:cemd_bin} (Section~\ref{sec:cemd}), one can extend the previous framework to hue distributions, seen as circular distributions. Following Equation~\eqref{eq:cemd_med}, the optimal mapping between the hue distribution $h_u$ of an image $u$ and the target hue distribution $h_t$ is obtained as $(H_t-\alpha)^{-1}\circ H_u$, where $\alpha$ is the median of the values $\{H_u[i] - H_t[i]\}$. Figure~\ref{fig:hue_transfer} illustrate this transfer of hue on a  pair of images. 
For a detailed survey on color transfer between images, we refer the reader to~\cite{fpitie07a}.

\begin{figure*}[htb]
   \begin{center}
  
   \includegraphics[width=\textwidth]{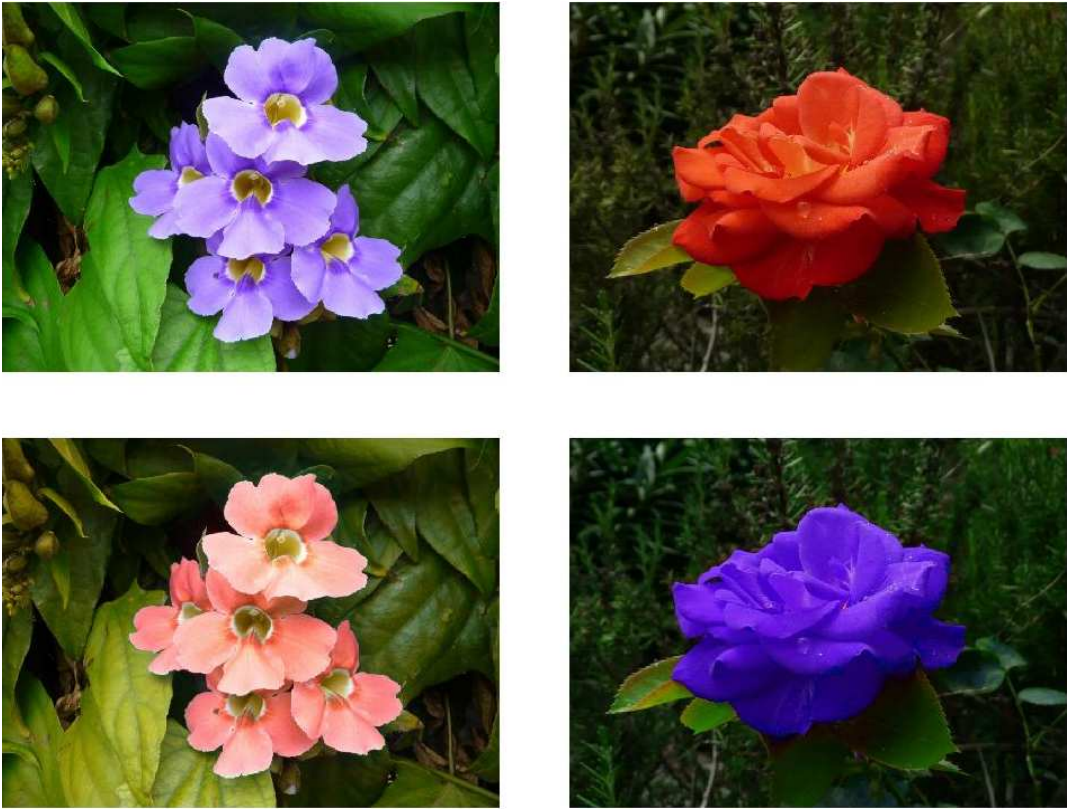}
  \end{center}
  
  \caption[Hue transfer]{\textbf{Hue transfer between images.} \emph{First row:} original images. \emph{Second row:} the hue channel of each image has been modified by applying the circular optimal transportation flow between the hue channels (see text for details), while other channels (``saturation'' and ``value'' in HSV representation) are kept unchanged.} 
  \label{fig:hue_transfer}
\end{figure*}

\subsection{Image and histogram comparison}\label{sec:hist_comparison}


In this section, we investigate the interest  of Monge-Kantorovich distances for histogram comparison.  More precisely, the distance considered in the following experiments is the $\textsc{MK}_1$ distance (given by Equation~\eqref{eq:cemd_med}) on the circle. Following~\cite{ref:rabin09siam}, we refer to it as \textsc{CEMD} (Circular Earth Mover's Distance). This distance is compared in particular to the $L^1$ distance, which can be described as being bin-to-bin, since it compares bins having the same index.

\subsubsection{Local feature comparison}\label{sec:feature}

Many computer vision tasks rely on local features (object recognition, image retrieval, indexation and classification, image mosaicing, \etc).
Some of the most commonly used (and invariant) local features are the Shape Context~\cite{shape_context} and SIFT descriptors~\cite{SIFT}, which both encode periodic data: polar orientation for the former, and gradient orientation for the latter.
For instance, a SIFT descriptor can be seen as a collection of $M$ one-dimensional and circular-histograms
, each one being constructed from a subpart of a localization grid in the image domain~\cite{SIFT}.

In~\cite{ref:rabin09siam}, it is demonstrated that the Circular Earth Mover's Distance (or $\textsc{MK}_1$ distance on the circle) is far more robust than classical bin-to-bin distances ($L^1$ and $L^2$ metric, $\chi^2$ distance, etc) to compare SIFT descriptors. In particular, it is underlined that this cross-bin distance is more adapted to two kinds of perturbations
 \begin{itemize}
 \item[$\bullet$] \textbf{quantization} effects (also know as the ``binning problem''), which result from the small number of samples used to build the discrete histograms of gradient orientations, but also from the localization grid used to extract histograms over the pixel grid;
 \item[$\bullet$] \textbf{shifts} in histograms, resulting from geometrical deformations in the image (\eg perspective effects, which typically arise when an object is seen from different points of view).
 \end{itemize}

Note that this result is consistent with the one presented in~\cite{pele2008eccv}, where the metric used to compare SIFT descriptors is a modification of the Monge-Kantorovich distance with a truncated cost $c$.
A more in depth analysis of the general robustness of Monge-Kantorovich distances for both quantization and shift effects is proposed in section~\ref{sec:comparative_study}.

\subsubsection{Three experiments on color image retrieval}\label{sec:color_retrieval}

For the task of color image retrieval, numerous studies have shown that the Earth Mover's Distance (defined in Section~\ref{sec:definitions}) often achieves better retrieval performances than bin-to-bin distances 
\cite{rubner_ijcv,EMD_image_retrieval_blobworld,tomasi_PAMI01,EMD_image_retrieval_blobworld2,Lv_retrieval,EMD_image_retrieval_semantic,EMD_image_retrieval_semantic2,Hurtut08,Pele_iccv2009}. In order to illustrate the advantages of the Circular Earth Mover's Distance in the same context, we rely on hue distributions to perform image retrieval on a color image dataset.
The dataset\footnote{the image dataset is available at the following address: 
\url{http://perso.telecom-paristech/~rabin/database/}} contains  $14$ category of $9$ pictures of the same object, with various camera settings (sensitivity, with or without flash, white balance reference, exposure time, \etc). Nine pictures of the same category are shown as an example in Figure~\ref{fig:base_photo_flash}.
Each of the $P=14\times 9 = 126$ images of the dataset is described by a hue (channel H of HSV representation) distribution, built on $N=360$ bins. For a given dissimilarity measure $D$, the retrieval experiment consists in using an image of the dataset as a query and finding the $r$ most similar images for $D$. For each value of $r$, we compute 
\begin{enumerate}
\item[$\bullet$] the \textbf{recall}, which is defined as the average, when the query spans the database, of the ratio between the number of correctly retrieved images among $r$ and the size of the query class;
\item[$\bullet$] the \textbf{precision}, which represents the average on the whole database of the rate of true positives among the $r$ most-similar images.
\end{enumerate}
The curves ($r$, recall) and (recall, precision) are drawn on Figure~\ref{fig:ROC_photo_flash}, for three different dissimilarity measures: \textsc{cemd} (Equation~\eqref{eq:cemd_med}), non circular \textsc{emd} (Equation~\eqref{eq:emd_bin}), and the $L^1$ distance.

\begin{figure*}[htb]
  \centering
  \includegraphics[width=\textwidth]{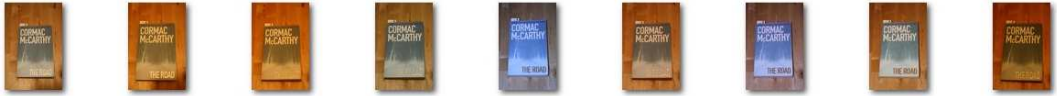}
  \caption{Example of a category of 9 pictures extracted from the image database used for image retrieval (results ate shown in Figure~\ref{fig:ROC_photo_flash}). These photographs represent the same scene under various illumination conditions and camera settings. }
\label{fig:base_photo_flash}
\end{figure*}

\begin{figure*}[htb]
  \centering
  \subfigure[Average {Recall} rate according to the rank]{\label{fig:rappel_cuisine_360}
  \includegraphics[width=.485\textwidth]{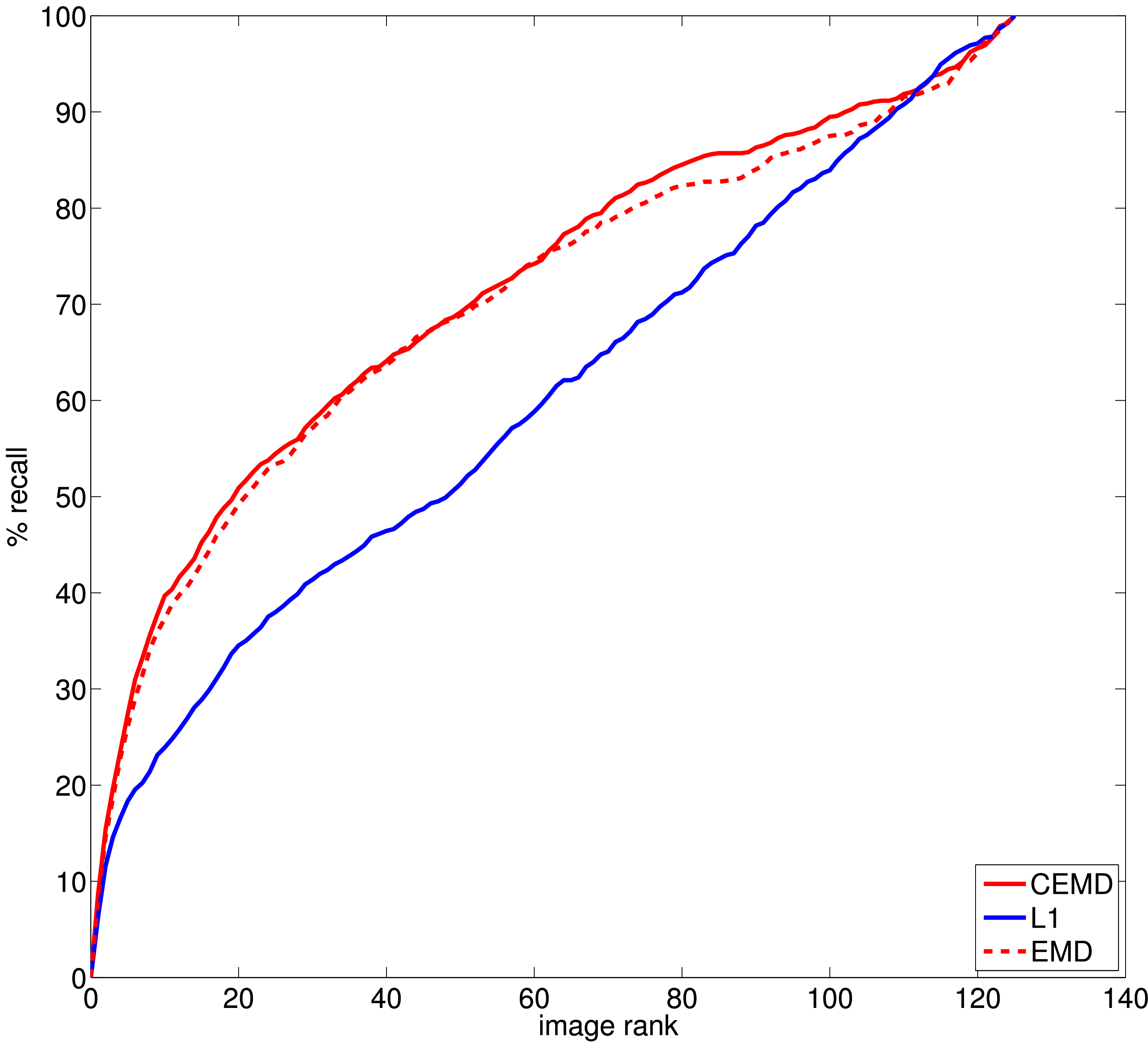} 
  }\hfill%
  \subfigure[Average {precision-recall} curve]{\label{fig:precision_cuisine_360}
  \includegraphics[width=.485\textwidth]{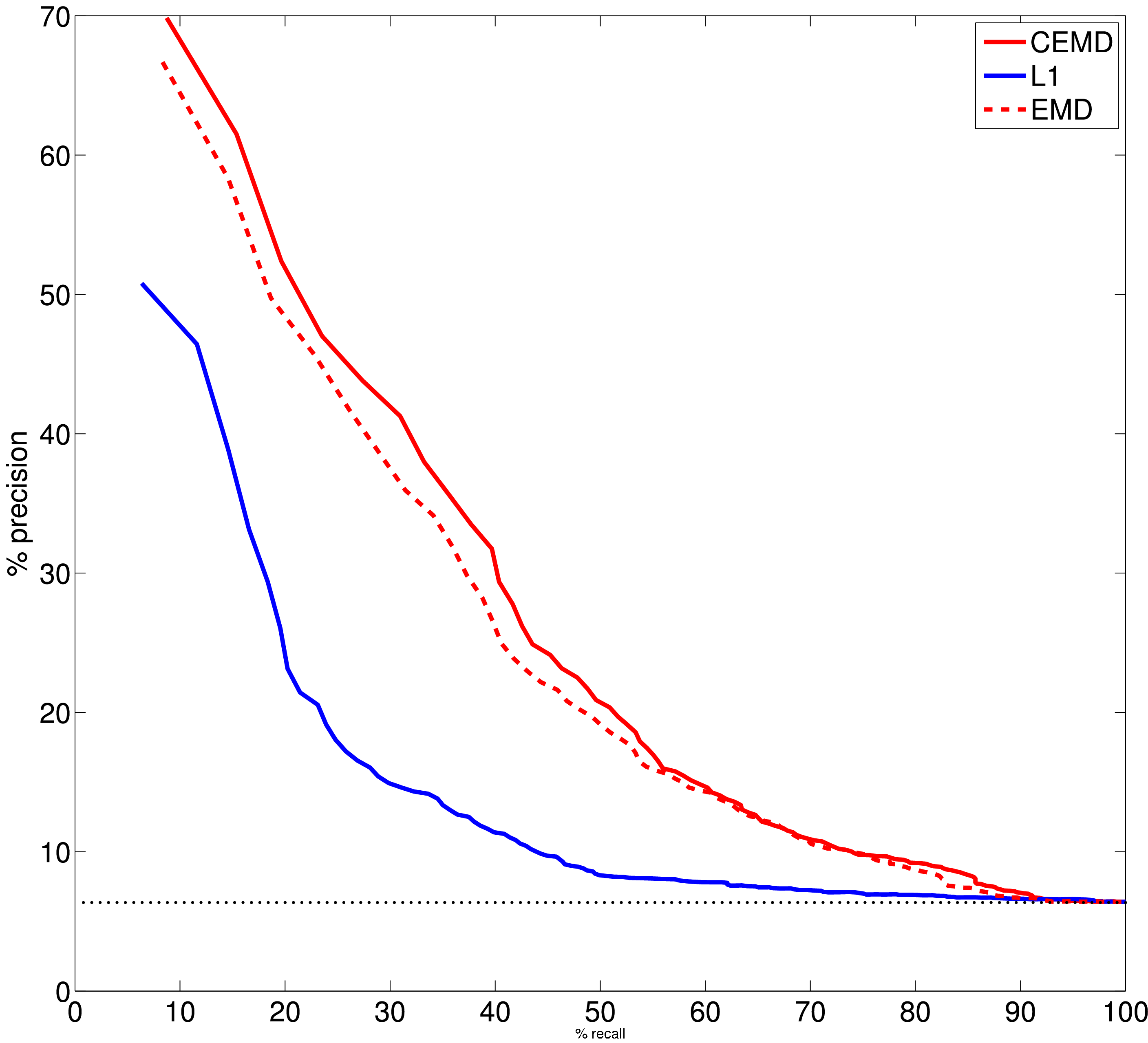}
  }
  
  \caption{\textbf{retrieval of a color image database.}
  The performances curves are obtained using \textsc{cemd} in red continuous line, $L^1$ distance in blue line and also EMD (non circular mass transportation) in red dashed line.}
\label{fig:ROC_photo_flash}
\end{figure*}

In this experiment, the results of \textsc{cemd} (in red continuous line) clearly outperforms those of $L^1$  and EMD (which does not take into account the circularity of hue histograms).
As in the experiments on local feature comparison, one can guess that the
superiority of Monge-Kantorovich distances are due to their natural robustness
to shifts in the distributions. However, observe that this time, in contrast
with the application to local descriptors, a huge number of samples ($\sim5.10^5$) is used to build each circular histogram, so that we avoid the aforementioned binning problem.


\paragraph{\bf \em  Two more color retrieval experiments} In this paragraph, we aim
at showing that two different classes of intraclass variability could arise when
representing data by histograms, and that bin-to-bin and cross-bin distances
behaves very differently according to these perturbations. This fact will then
be discussed in more detail in Section~\ref{sec:comparative_study}.

In order to illustrate these phenomena, a small dataset\footnote{the image dataset is available at the following address: 
\url{http://perso.telecom-paristech/~rabin/database/}} of 22 photographs has been used, shown in Figure~\ref{fig:base_photo_julie}.
Here, we propose to reproduce --in a synthetic way-- the color image retrieval experiment presented in section~\ref{sec:color_retrieval}. 
For each picture of this dataset, synthetic modifications are proceeded in order to simulate two types of perturbations which naturally arise when considering color image retrieval:
\begin{itemize}
\item[$\bullet$] \textbf{Gamma} correction with a power factor varying from 0.6 to 2.4 (this operation has been realized on the ``Luminance'' channel in CIELab color space). 
An example is shown in Figure~\ref{fig:Planche_julie_gamma};

\item[$\bullet$] \textbf{White balance} correction with a ``color temperature'' varying from $4400$ to $6200^\circ K$  
(Example is given in Figure~\ref{fig:Planche_julie_white_balance}).
\end{itemize}

\begin{figure*}[htb]
  \centering
  \includegraphics[width=\textwidth]{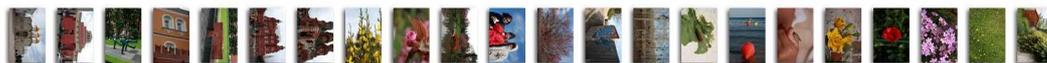}
  \caption{22 pictures used for image retrieval test (results are shown in Figures~\ref{fig:ROC_gamma} and~\ref{fig:ROC_whitebalance}).}
\label{fig:base_photo_julie}
\end{figure*}

Now, applying these modifications to the dataset, we obtain two different databases on which a retrieval is performed using $L^1$ and {\sc cemd} metric (like in section~\ref{sec:color_retrieval}).
Results are shown in Figure~\ref{fig:ROC_gamma} for gamma correction, and in Figure~\ref{fig:ROC_whitebalance} for color temperature correction.


\begin{figure*}[htb]
  \centering
  \subfigure[Example of gamma modification]{\label{fig:Planche_julie_gamma}
   \includegraphics[width=\textwidth]{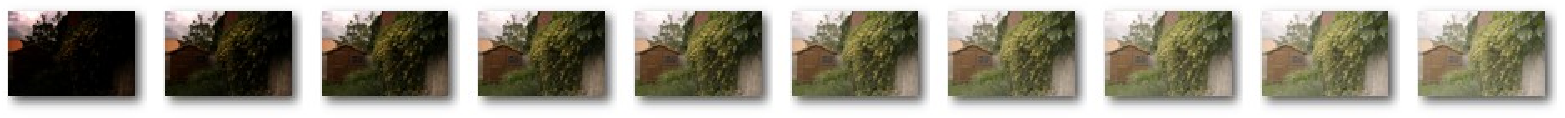}
  }
  
  \centering
  \subfigure[Average recall rate]{\label{fig:rappel_gamma}
   \includegraphics[trim = 0 0 40 0,clip, angle=0,width=.47\textwidth]{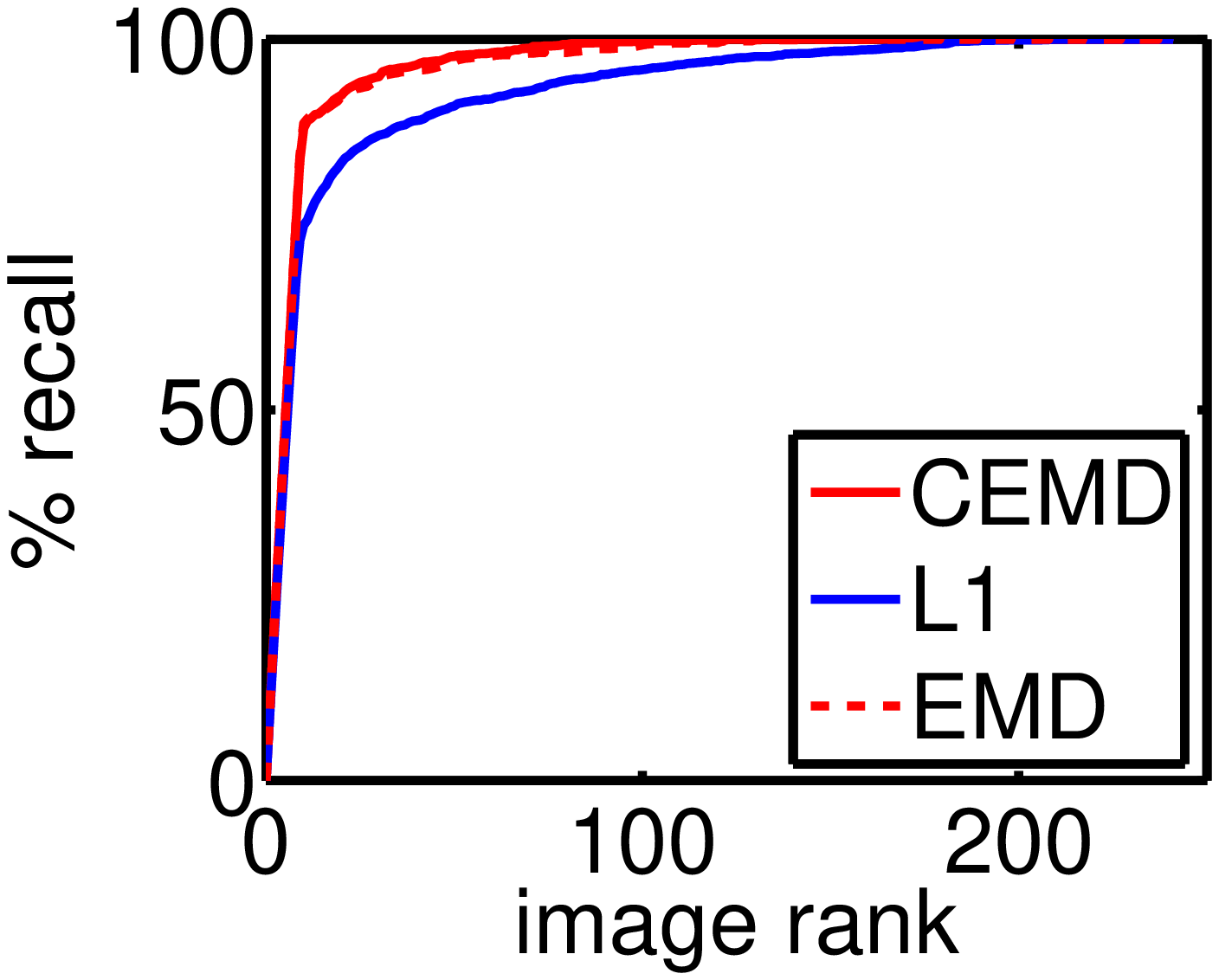}
  }\hfill%
  \subfigure[Precision-recall curve]{\label{fig:precision-rappel_gamma}
   \includegraphics[trim = 0 0 20 0,clip, angle=0,width=.50\textwidth]{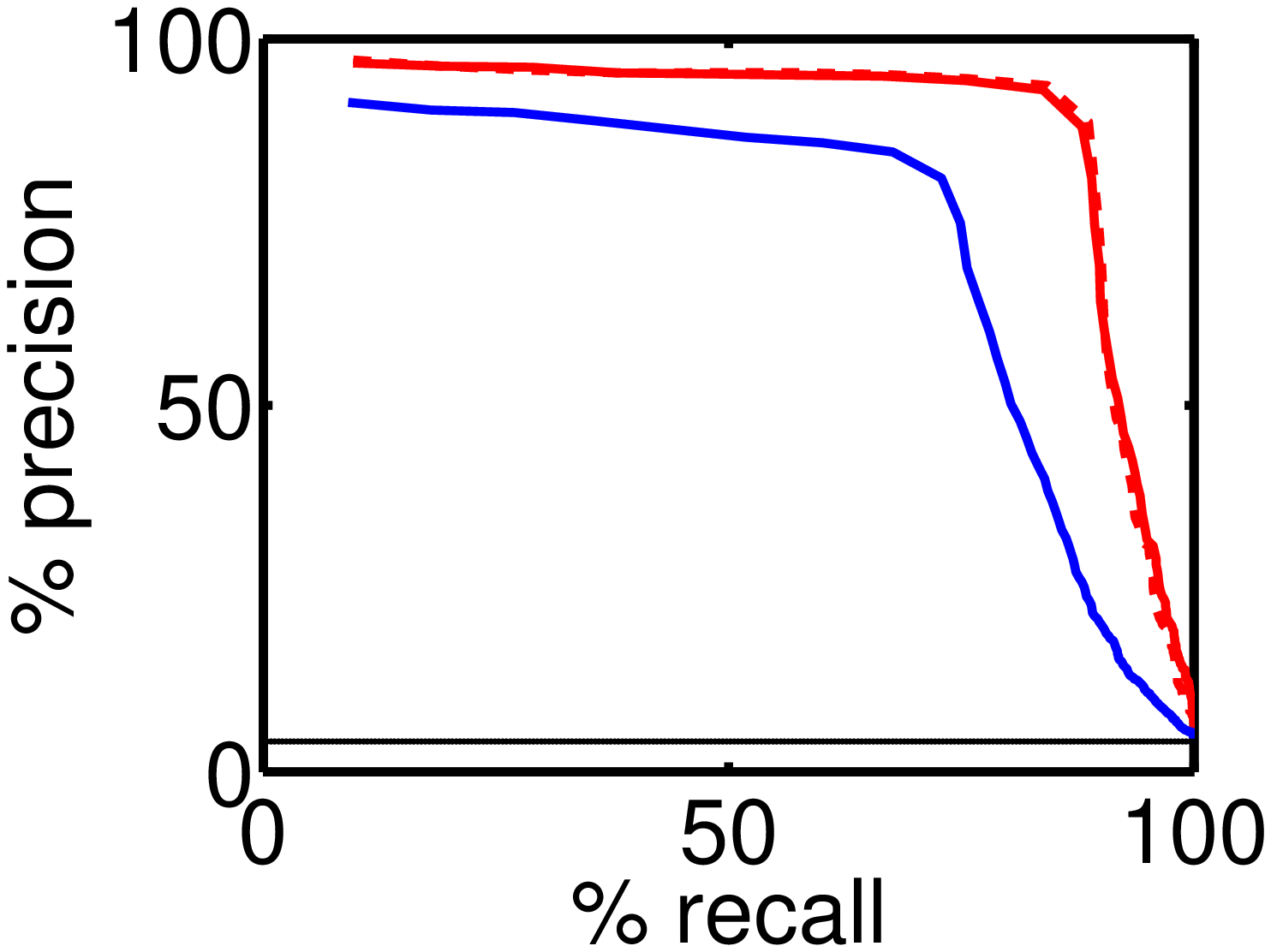}
  }
\caption{Retrieval results with data corrupted by contrast modification (gamma correction).
  }
\label{fig:ROC_gamma}
\end{figure*}

Once again, in the case of gamma correction (Figure~\ref{fig:ROC_gamma}) {\sc cemd} provides in average better retrieval results than $L^1$ distance.
The main reason in such case is that we observe some \textbf{intraclass shifts} between histograms, for which cross-bin distances such as {\sc cemd} are more robust than bin-to-bin distances.

\begin{figure*}[htb]
  \centering
  \subfigure[Example of color temperature modification]{\label{fig:Planche_julie_white_balance}
   \includegraphics[width=\textwidth]{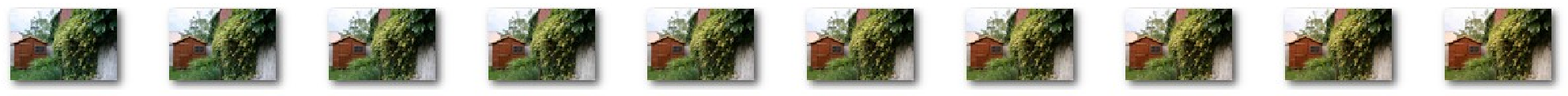}
  }
  
  \subfigure[Average recall rate]{\label{fig:rappel_whitebalance}
  \includegraphics[trim = 0 0 40 0,clip,width=.47\textwidth]{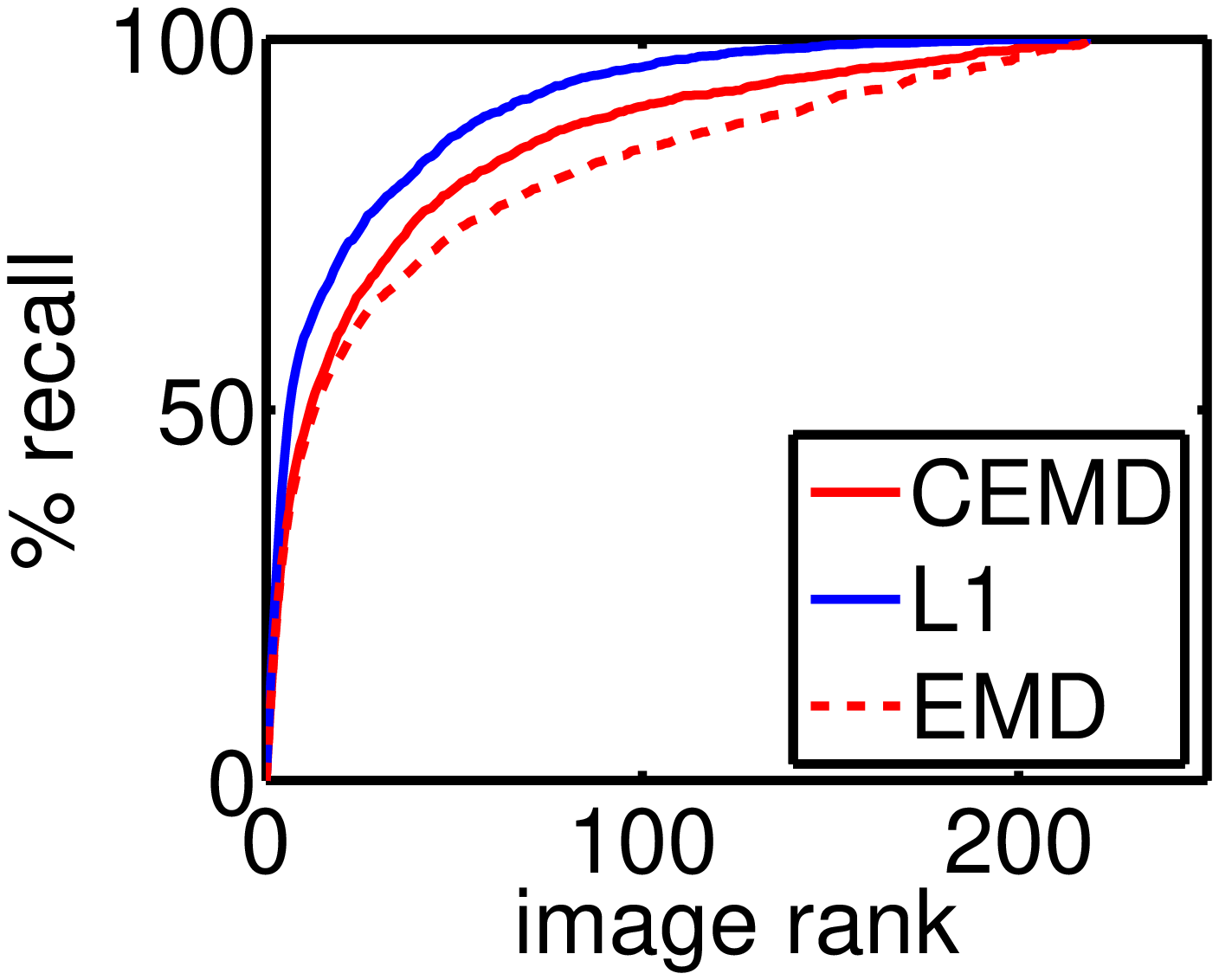}
  }
  \subfigure[Precision-recall curve]{\label{fig:precision-rappel_whitebalance}
  \includegraphics[trim = 0 0 20 0,clip,width=.50\textwidth]{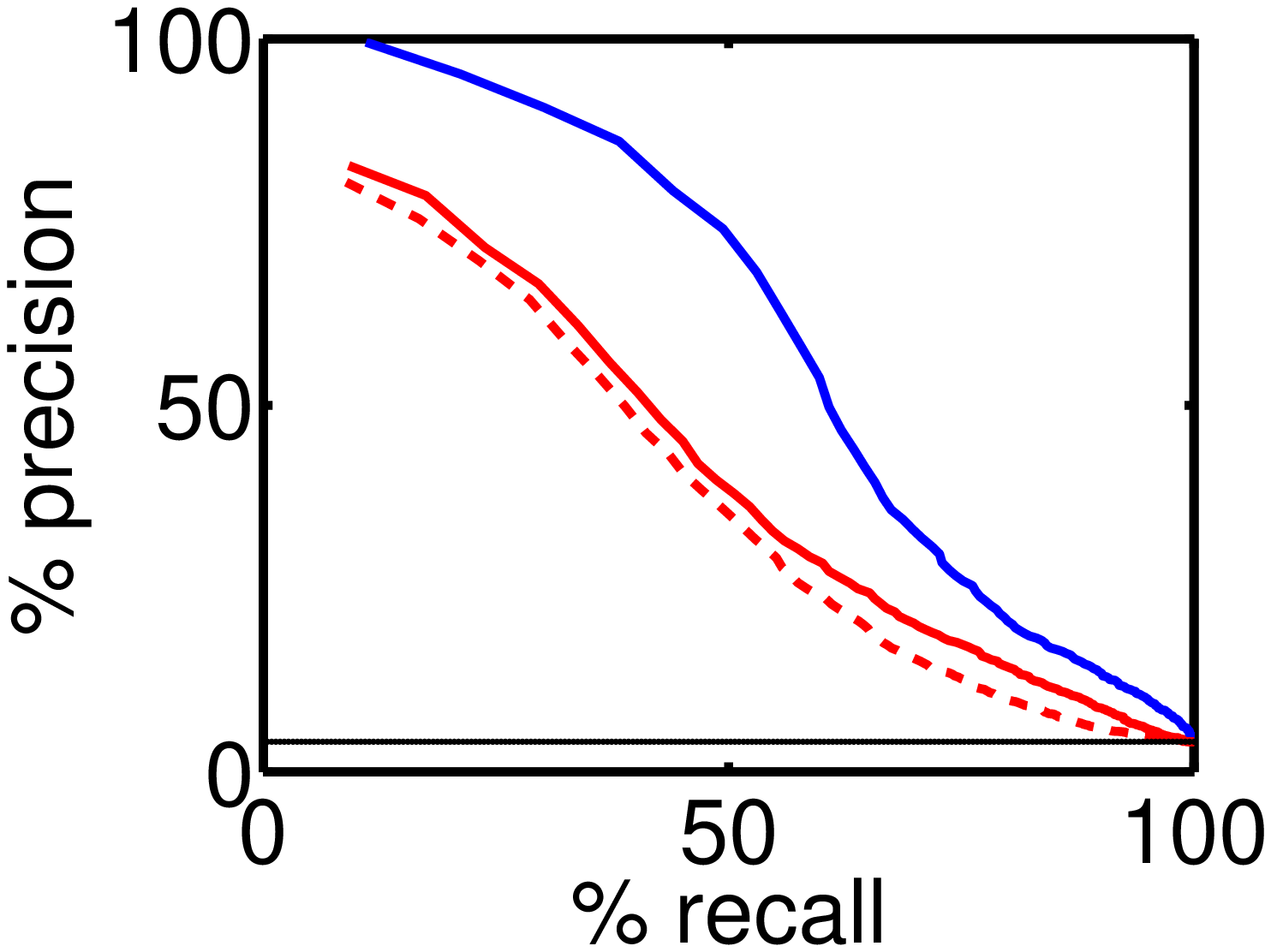}
  }
\caption{Retrieval results with data corrupted by white balance modification (color temperature correction).}
\label{fig:ROC_whitebalance}
\end{figure*}

Now, in the case of color temperature modification (Figure~\ref{fig:ROC_whitebalance}), one observes the following result: $L^1$ distance provides better retrieval scores than {\sc cemd}.
An examination of the results has led us to observe that, in such a case, the intraclass variability results this time from \textbf{differences of weight} of dominant modes in histograms (see Figure~\ref{fig:weight} for an illustration).

\begin{figure*}[htb]
  \centering
  \centering
  \subfigure[Illustration of histogram shift]{\label{fig:shift}
    \hspace{0.01\textwidth}
    \includegraphics[width=.35\textwidth]{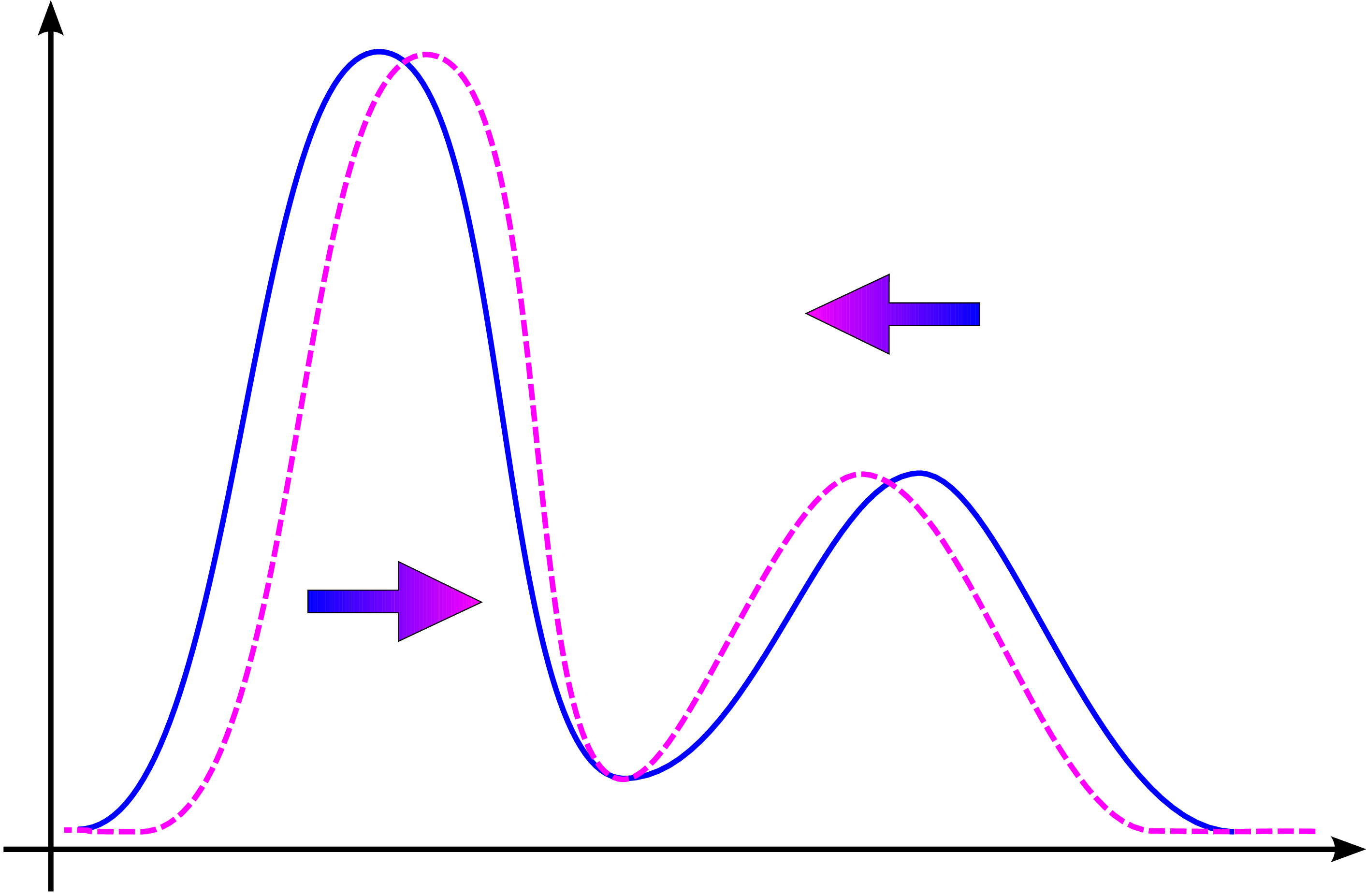}
    \hspace{0.01\textwidth}
  }
  \hspace{0.05\textwidth}
  \subfigure[Illustration of histogram weight variability]{\label{fig:weight}
      \hspace{0.01\textwidth}
      \includegraphics[width=.35\textwidth]{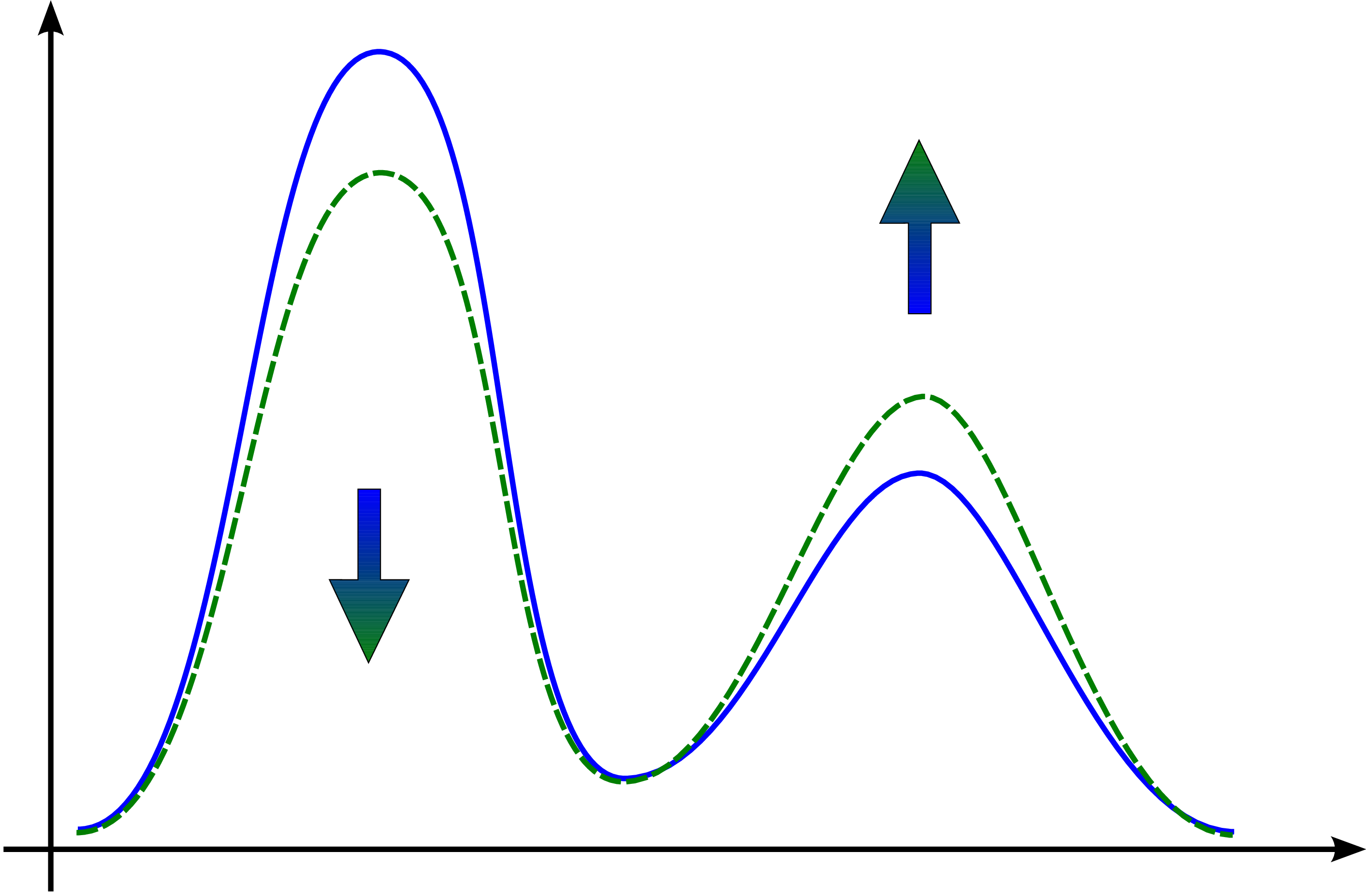}
      \hspace{0.01\textwidth}
  }
\caption{Illustration of the two main classes of perturbations involved in retrieval performances: intraclass shift variability (to the left) and intraclass weight variability (to the right).}
\label{fig:schema_perturbation}
\end{figure*}

\smallskip
In order to understand the implications of these results, a discussion is proposed in the following section.

\section{Is it worth using  transportation distances to compare histograms ?}
\label{sec:comparative_study}

Following the last two experiments of the previous section (gamma correction
and color balance), this section provides a discussion on the relative advantages of Monge-Kantorovich
distances using convex cost functions, those using concave cost functions,
and bin-to-bin distances. The discussion is not specific to the circular case
and will be made from non-circular synthetic examples. 

Writing as before $d(x,y)$ for the geodesic distance on the circle, we consider the following distances: 
\begin{itemize}
\item[$\bullet$] the $L^1$ bin-to-bin distance,
\item[$\bullet$] Monge-Kantorovich distances with concave cost functions:
  \begin{itemize}
      \item[$\circ$] $\textsc{MK}_{\exp\tau}$ defined from Formula~(\ref{eq:monge-kantorovich_c})
      when using the exponential cost function $c(x,y)=1-\exp(-\frac{d(x,y)}{\tau})$
      \item[$\circ$] $\textsc{MK}_{T \tau}$ 
      defined from Formula~(\ref{eq:monge-kantorovich_c}) when using a thresholded cost function as introduced in ~\cite{pele2008eccv,Pele_iccv2009}, that is, $c(x,y)=\min(d(x,y),\tau)$ (see
      Section~\ref{sec:thresholded}).
  \end{itemize}
\item[$\bullet$] Monge-Kantorovich distances with convex cost functions:
  $\textsc{MK}_\lambda=(\textsc{MK}_c)^{\frac 1 \lambda}$, with $\textsc{MK}_c$ the quantity defined by
  Formula~(\ref{eq:monge-kantorovich_c}) when using a cost fonction $c(x,y)=d(x,y)^\lambda$, for $\lambda\geq 1$.
\end{itemize}

Recall that among these distances, only $L^1$, $\textsc{MK}_\lambda$ and $\textsc{MK}_{T 2}$
can be computed in linear time. Observe also, following the
remarks of Section~\ref{sec:thresholded} on the proximity between $\textsc{MK}_{T 2}$
and $L^1$, that these distances in a sense
produce a complete range of alternatives between bin-to-bin distances (such as
$L^1$) and Monge-Kantorovich associated with highly non-convex cost functions
(e.g. $\textsc{MK}_3$). This fact will be quite clear in the following synthetic experiments.

These experiments consist, in order to study the assets of the various
distances, to perform retrieval from synthetic histograms (mixture of two
Gaussians) in the presence of two types of
perturbations: shifts in the positions of bins on the one hand, and variation
in the weight of bins on the other hand (see Figure~\ref{fig:schema_perturbation}). Observe that these two types of perturbations correspond to the
ones encountered at the end of Section~\ref{sec:color_retrieval}.

We assume that elements to be compared belong to two classes $A$ and $B$, and
that each element is represented by one $N$-bins histogram. We model the histograms as the
mixture of two Gaussians. Writing $c \in \{A,B\}$ for the class, these two Gaussians have weights $p^c$ and $(1-p^c)$, means
$\mu_1^c$ and $\mu_2^c$, and standard deviations $\sigma_1^c$ and $\sigma_2^c$
(see Figure~\ref{fig:melange_gaussienne}). In the following experiments,
parameters are set as follows 
\begin{itemize}
\item[$\bullet$] \emph{Histogram construction} Quantization of histograms: $N=100$ bins; Number of samples for Gaussian mixture data generation: $1,000$ samples in $[0,1]$; Number of histograms per class: $1,000$ histograms.
\item[$\bullet$] \emph{Gaussian mixture parameters} Weights: $p^A=0.6$ and $p^B=0.8$; Means: $\mu_2^A=\mu_2^B=0.2$ and $\mu_2^A=\mu_2^B=0.7$; Standard-deviations: $\sigma_1^A=\sigma_1^B=\sigma_2^A=\sigma_2^B=0.05$.
\end{itemize}

\begin{figure*}[htb]
  \centering
  \includegraphics[angle=0,width=0.9\textwidth]{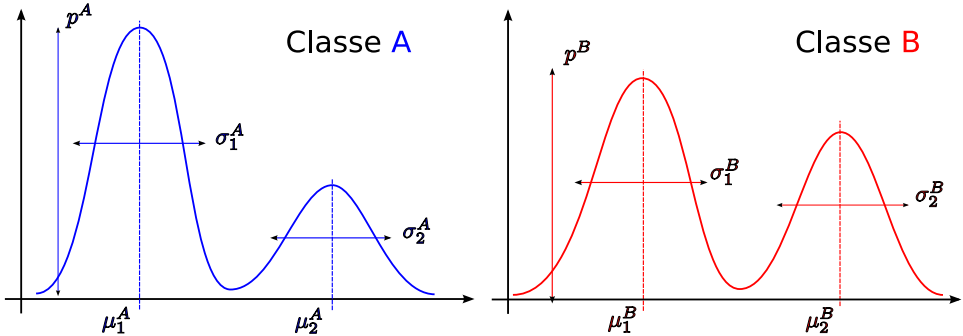}
  \caption{The two classes $A$ and $B$ are defined as a Gaussian mixture model. For each class, the two Gaussian distributions are defined with 4 parameters (means and standard-deviations), plus a weighting parameter $p$.}
  \label{fig:melange_gaussienne}
\end{figure*}

This generative model being chosen, two different kinds of variability can now
be simulated to evaluate the robustness of transportation distances depending
on the cost function~\footnote{The Earth Mover's Distances with exponential
  and truncated cost functions have been computed using the code kindly
  provided by Y. Rubner~\cite{emd_code}.} (see Figure~\ref{fig:schema_perturbation}). 

\paragraph{\bf \em  Histogram shift}

We introduce random shifts in the histogram by modeling the means $\mu_1^c$ as
random variables. We choose $\mu_1^A = 0.2 + \epsilon_\mu$, where $\epsilon_\mu$ is uniformly drawn in $[-0.1;0.1]$.
Some of such generated histograms are superposed in Figure~\ref{fig:exp_shift}.
The precision-recall curves resulting from this two-class retrieval problem are plotted for different metrics in Figure~\ref{fig:ROC_shift}.
One first observes that distances $\textsc{MK}_\lambda$, relying on convex cost functions, give the best results, the larger $\lambda$ the better. Second, it can be seen that transportation distances with concave cost function yields less efficiency. First are distances relying on an exponential cost.  Eventually using transportation distances with truncated $L^1$ distances provides poor results, similar to those obtained with the $L^1$ distance. This fact is in agreement with the analysis made in Section~\ref{sec:thresholded}.

\begin{figure*}[htb]
  \centering
  \subfigure[Histogram shift variability experiment]{\label{fig:exp_shift}
      \includegraphics[angle=0,width=0.48\textwidth]{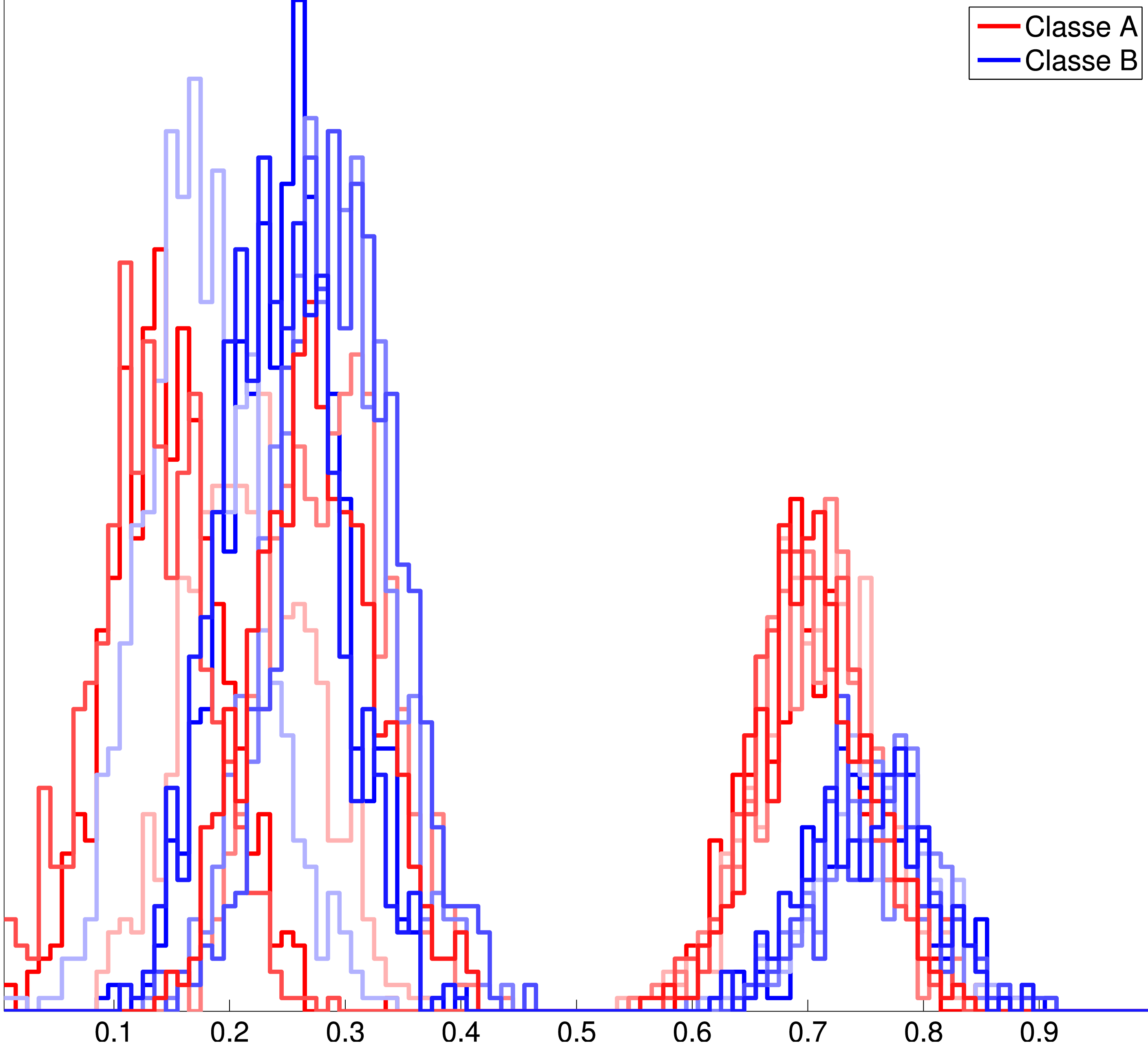}}
  \subfigure[Precision-Recall curve]{\label{fig:ROC_shift}
  \includegraphics[angle=0,width=0.48\textwidth]{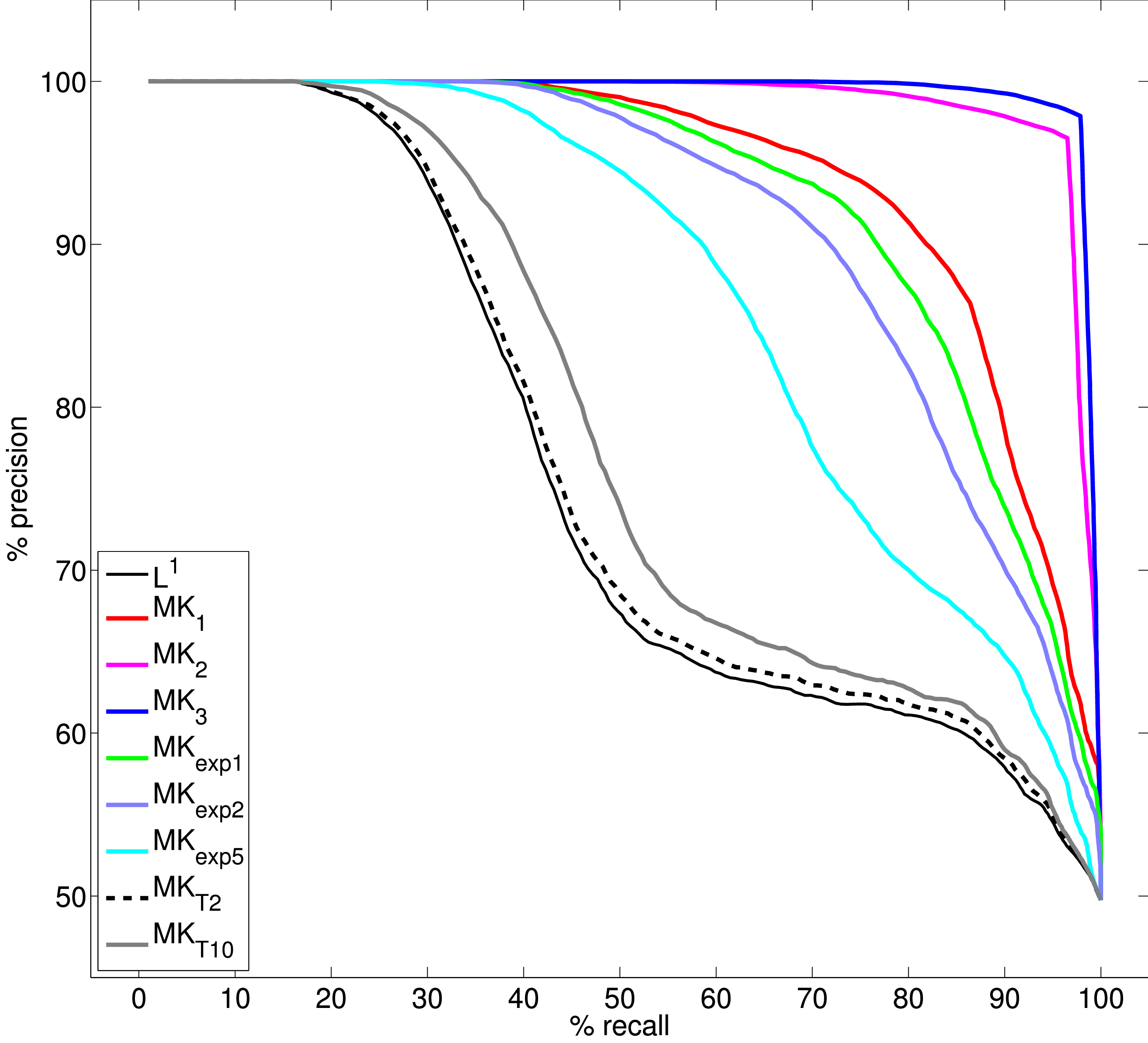}}
  \caption{\textbf{Two-class retrieval problem with intraclass shift variability}.
  The effect of the perturbation on histograms is shown in Figure~\ref{fig:exp_shift}.
  The Precision-Recall curves are displayed in Figure~\ref{fig:ROC_shift} for several transportation distances: $\textsc{MK}_{T \tau}$ refers to as the transportation distance with truncated cost function according to the threshold $\tau \in \{2,10\}$, $\textsc{MK}_{\exp\tau}$ corresponds to the transportation distance with exponential cost function using parameter $\tau \in \{1,2,5\}$, and $\textsc{MK}_{\lambda}$ is the Monge-Kantorovich distance with $\lambda \in \{1,2,3\}$. In addition is shown the curve obtained with $L^1$ metric, which is equivalent to $\textsc{MK}_{T 1}$ (see \S~\ref{sec:L1_Monge}).
  } 
  \label{fig:perturbation_mean}
\end{figure*}

\paragraph{\bf \em Histogram weight variability}

In the second experience, intraclass weight variability are now simulated by
modeling weights as random variables: $p_1^A = 0.6 + \epsilon_p$, where $\epsilon_p$ is uniformly drawn from $[-0.1;0.1]$.
Some of such generated histograms are superposed in Figure~\ref{fig:exp_weight}.
The precision-recall curves resulting from this two-class retrieval problem are plotted for different metric in Figure~\ref{fig:ROC_weight}.
One observes that with this kind of perturbation, transportation distances
with $L^1$ cost function are less robust than the $L^1$ distance.
This time, it can be seen that distances with concave cost function yield
better retrieval performances. Using  thresholded cost functions again
provides results that are very similar
to those obtained with the $L^1$ distance. In the meantime, distances relying
on exponential cost functions are still half-way between convex cost functions
and thresholded cost functions.

It therefore appears that higher robustness to one type of perturbation
yields poorer robustness to the other type. There is a logical tradeoff between
robustness to shifts and weight variability. In this context, and given that it may be computed in linear time, the $\textsc{MK}_1$
distance appears as a good compromise in term of computational cost and
robustness to the two kinds of variability considered here.

\begin{figure*}[htb]
  \centering
  \subfigure[Histogram weight variability experiment]{\label{fig:exp_weight}
   \includegraphics[angle=-0,width=0.48\textwidth]{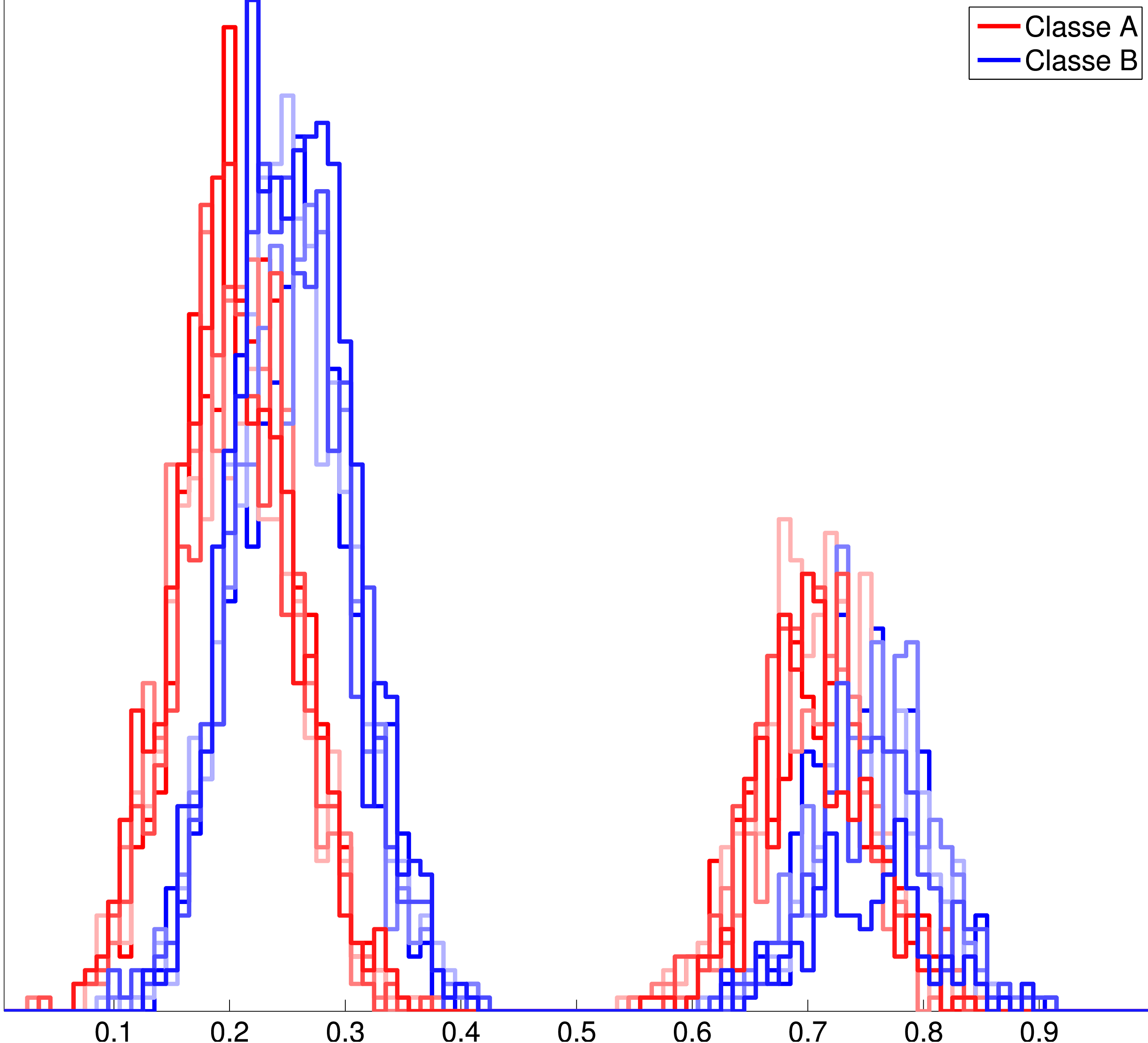}}
  \subfigure[Precision-Recall curve]{\label{fig:ROC_weight}
  \includegraphics[angle=0,width=0.48\textwidth]{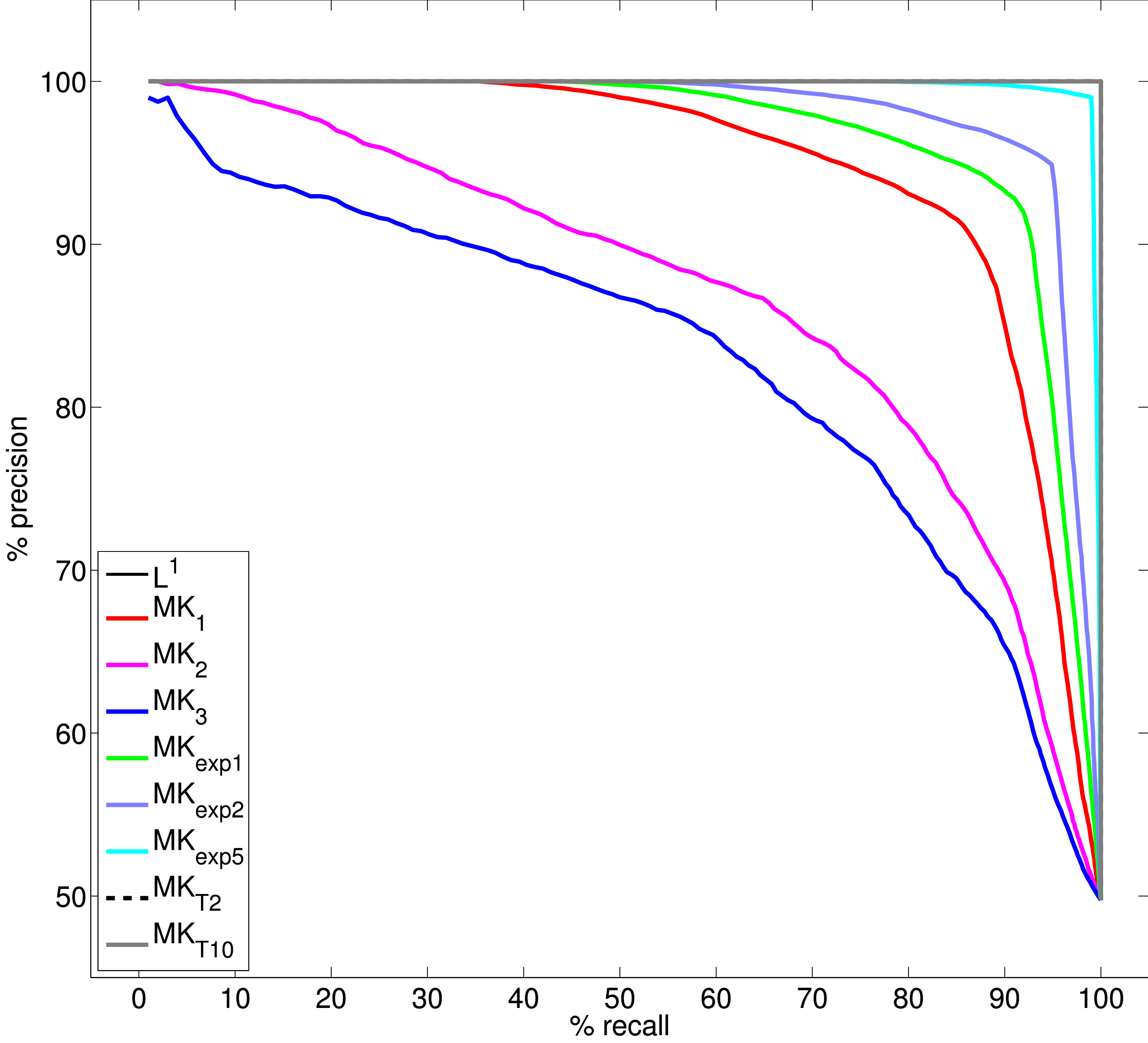}}
  \caption{\textbf{Two-class retrieval problem with intraclass weight variability}.
  The effect of the perturbation on histograms is shown in Figure~\ref{fig:exp_weight}.
  The Precision-Recall curves are displayed in Figure~\ref{fig:ROC_shift}, plotted for different transportation distances.}
   \label{fig:perturbation_weigth}
\end{figure*}

\section{Conclusion}\label{sec:conclusion}

In this paper, the optimal mass transport problem on the circle has been addressed in the case of convex and increasing cost function of the geodesic distance on the circle.
We have proposed a new formulation (and a proof) for estimating the corresponding Monge-Kantorovich distances.
In the particular case where the cost function is the geodesic distance on the circle, it has been shown that the transportation distance $\text{MK}_1$ between circular histograms (also referred to as \textsc{cemd}, standing for Circular Earth Mover's Distance) can be deduced by a very simple Formula~\eqref{eq:cemd_med} which is computed in linear time.

Then, several applications in this framework has been studied (hue transfer, local features comparison and color image retrieval), exploiting both the optimal transportation cost between histograms but also the corresponding optimal flow.
Other applications could also benefit from the \textsc{cemd} metric, such as shape recognition based on circular descriptors (see \eg character recognition with orientation histogram~\cite{histogram_orientation_emd}, and curvature based descriptor along closed contour~\cite{Mokhtarian_curvature_shape,morphological_curvature_shape}).

In the last section, a comparative analysis of transportation distances with different cost functions has been proposed, 
considering two types of perturbations which arise with histogram representation: mean and weight changes of dominant modes. 
We have demonstrated that there is a tradeoff between these two phenomena when using either convex or concave cost functions.
Eventually, the proposed \textsc{cemd} metric offers an interesting compromise between these two choices, while being easy to use.

\begin{acknowledgements}
Delon acknowledges the support of the French Agence Nationale de la Recherche (ANR), under grant BLAN07-2\_183172, Optimal transport: Theory and applications to cosmological reconstruction and image processing (OTARIE), and would like to thank J. Salomon and A. Sobolevski for fruitful discussions.
\end{acknowledgements}

\begin{appendix}
\begin{normalsize}

\section{Appendix: Proof of Theorem~\ref{theoremMK}}\label{sec:annexe}

This appendix provides a complete proof of Theorem~\ref{theoremMK} in the case where $f$ and $g$ are discrete distributions (as written in Equation~\eqref{eq:discrete_histograms}).  We first prove this theorem for distributions composed of unitary masses, and conclude thanks to continuity arguments.

\subsection{Introduction}

Consider two discrete sets of points $\{x_1,\dots x_P\}$ and $\{y_1,\dots y_P\}$ on the unit circle $S¹$, and the corresponding discrete distributions 
\begin{equation}
\label{eq:unitary_masses}
f = \frac{1}{P} \sum_{k=1}^P \delta_{x_k}, \text{ and }  g = \frac{1}{P} \sum_{k=1}^P \delta_{y_k},
\end{equation}
where the notations $x_k$, $y_k$ are used equally for points on the unit circle or for their coordinates in $[0,1[$. Let $d$ be the geodesic distance along the circle (given by Equation~\eqref{eq:distance}) and assume that $c$ can be written $c(x,y) = h(d(x,y))$ with $h$ a nonnegative, increasing and convex function.
It is well known (this is a consequence of Birkhoff's theorem, see for example the introduction of~\cite{ref:villani03book}) that the optimal transportation cost between $f$ and $g$ equals
\begin{equation}
\label{distance_MKsigma}
  \textsc{MK}_{c}\,(f,g)=\min_{\sigma \in \Sigma_P}   W_{\sigma}^{c} \,(f,g),\;\text{ with }  \;\; W_{\sigma}^{c}\,(f,g) : = \frac{1}{P} \sum_k c(x_k, y_{\sigma(k)}) = \frac{1}{P} \sum_k h(d(x_k, y_{\sigma(k)})),
\end{equation}
 where $\Sigma_P$ is the set of permutations of $\{1,\dots P\}$. 
In other words, finding the optimal transportation between $f$ and $g$ boils down to find the optimal permutation $\sigma$ between the points $\{x_k\}$ and $\{y_j\}$.

\subsubsection{Paths}\label{sec:paths}

If $x$ and $y$ are two different points of $S^1$, we note \textbf{$\gamma(x,y)$ the geodesic path linking $x$ and $y$ on $S^1$} (the path is supposed open: it does not contain $x$ and $y$). This path is always unique except in the case where $x$ and $y$ are in opposite positions on the circle. In this case, we choose $\gamma(x,y)$ as the path going from $x$ to $y$ in the trigonometric direction. 
A path $\gamma(x,y)$ is said to be \textbf{positive} if it goes from $x$ to $y$ in the trigonometric direction. If the path goes from $x$ to $y$ in the opposite direction , it is said to be \textbf{negative}.

\subsubsection{Cumulative distribution functions}\label{sec:cumulative}
The cumulative distribution function of $f$ has been defined in Equation~\eqref{eq:cumulative}. Now, on $[0,1[$ seen as a the unit circle $S¹$, no strict order can be defined between points, which means that we can define as many cumulative distribution functions as there are starting points on the circle.
If  $x$ is a point in $[0,1[$,  the $x$-cumulative distribution function $F_x$ of $f$ can be defined by choosing $x$ as the reference point on the circle $S^1$ and by summing the mass in the trigonometric order from this new reference point: 
\begin{equation}
  \label{eq:cumulative_x}
 \forall y \in \R,\;\;  F_x(y)=  F(x+y) - F(x).
\end{equation} 
An example of a cumulative distribution $F$ and its corresponding $x$-cumulative distribution $F_x$ on $[0,1[$ is shown on Figure~\ref{fig:xcumulative}.

\begin{figure}[htb]
  \centering
\includegraphics[width=8cm]{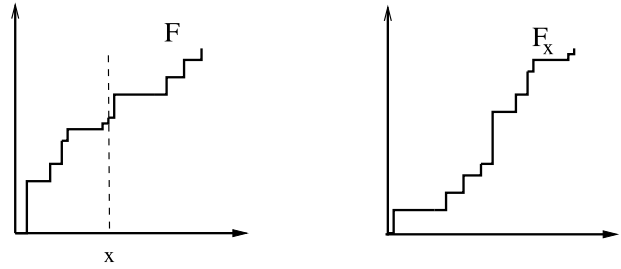}
\caption{$F$ on the left and $F_x$ on the right.}
\label{fig:xcumulative}
\end{figure}

\subsection{Preliminary results}

In the following, we prove that if $f$ and $g$ can be written as in Equation~\eqref{eq:unitary_masses}, if the points $x_1,\dots x_P$ and $y_1,\dots y_P$ are pairwise different, and if $\sigma$ is an optimal permutation for~(\ref{distance_MKsigma}), there is always a point on the circle which is not contained in any optimal path of $\sigma$. This result is proven first for strictly convex functions $h$ and for any optimal permutation $\sigma$, then for convex functions $h$ and a well chosen optimal permutation. 

\begin{proposition}
\label{prop1}
\textit{Assume that $h$ is strictly convex. Let $x_1, \dots x_P$ and $y_1,\dots y_P$ be $P$ points in $[0,1[$, all pairwise different. 
Then for each permutation $\sigma$ of $\Sigma_P$ which minimizes~(\ref{distance_MKsigma}) , there exists $k \in \{1, \dots P\}$ such that for all $l \neq k$, $x_k \notin \gamma(x_l,y_{\sigma(l)})$.}
\end{proposition}

The proof of this proposition needs the following lemma, which describes some properties of the geodesic paths $\gamma(x_l,y_{\sigma(l)})$ obtained when $\sigma$ is a minimizer of  (\ref{distance_MKsigma}) and $h$ is strictly convex.

\begin{lemma}
\label{lemma_path}
\textit{Assume that $h$ is strictly convex. Let $\sigma$ be a minimizer of  (\ref{distance_MKsigma}) and let $\gamma_l=\gamma(x_l,y_{\sigma(l)})$ and $\gamma_k=\gamma(x_k,y_{\sigma(k)})$ (with $l \neq k$) be two geodesic paths for the assignment defined by $\sigma$. Assume also that $x_l \neq x_k$ and $y_{\sigma(l)} \neq y_{\sigma(k)}$. Then, one of the following holds:
\begin{itemize}
\item[$\bullet$] $\gamma_l \cap \gamma_k = \emptyset$ ;
\item[$\bullet$] $\gamma_l \cap \gamma_k \neq \emptyset$ and in this case $\gamma_l$ and $\gamma_k$ have the same direction (both positive or both negative) and neither of them is contained in the other.
\end{itemize}
}\end{lemma}

\begin{proof}
  Assume that $\gamma_l \cap \gamma_k \neq \emptyset$. If  $\gamma_l \cap \gamma_k$ is equal to $\gamma(x_l,x_k)$, then, since $h$ is an increasing function of $d$,  $ c(x_l,y_{\sigma(l)})> c(x_k,y_{\sigma(l)})$ and $c(x_k,y_{\sigma(k)})> c(x_l,y_{\sigma(k)})$, which contradicts the optimality of $\sigma$. 
The same conclusion holds if  $\gamma_l \cap \gamma_k$ is equal to $\gamma(y_{\sigma(l)},y_{\sigma(k)})$.
Moreover, if for example the path $\gamma_l$ is included in $\gamma_k$, then the strict convexity of the function $h$ implies
$$c(x_l,y_{\sigma(l)})+c(x_k,y_{\sigma(k)})
> c(x_l,y_{\sigma(k)})+c(x_k,y_{\sigma(l)}),$$ which also contradicts the optimality of $\sigma$. 
 Thus, $\gamma_l \cap \gamma_k$ is equal to $\gamma(x_l,y_{\sigma(k)})$ or to $\gamma(x_k,y_{\sigma(l)})$ and it follows that $\gamma_k$ and $\gamma_l$ are either both positive or both negative.
\end{proof}

\begin{proofof}{Proposition~\ref{prop1}}
Let $\sigma$ be a minimizer of~(\ref{distance_MKsigma}). In the following, we will denote by $\gamma_l$ the geodesic path $\gamma({x_l,y_{\sigma(l)}})$. We can assume without loss of generality that the points $x_1, \dots x_P$ are in trigonometric order on the circle.

Assume that for each $l \in \{1,\dots P \}$, there exists $q(l) \neq l $ such that $x_l$ belongs to the open path $\gamma_{q(l)}$.
Then, for each $l$, we have $\gamma_{q(l)} \cap \gamma_{l} \neq \emptyset$, which means that the geodesic paths $\gamma_{q(l)}$ and $\gamma_{l}$ are either both positive or both negative (from lemma~\ref{lemma_path}).
Assume for instance that they are both positive and let us show that in this case $x_l \in \gamma_{l-1}$ (with $l-1=P$ if $l=0$). If $ q(l) = l-1$, there is nothing to prove. If $ q(l) \neq l-1$, it means in particular that $x_{q(l)}, x_{l-1}, x_l$ are in trigonometric order on the circle. Since $\gamma_{q(l)}$ is a positive path starting from  $x_{q(l)}$ and containing $x_l$, it follows that $\gamma_{q(l)}$ contains $x_{l-1}$ (recall that the points are assumed to be pairwise different, in particular $x_{l-1} \neq x_{q(l)}$). Thus $\gamma_{l-1} \cap \gamma_{q(l)} \neq \emptyset$, which implies that $\gamma_{l-1}$ is positive. Now, $x_l$ must be in $\gamma_{l-1}$, otherwise we would have $\gamma_{l-1} \subset \gamma_{q(l)}$, which contradicts lemma~{\ref{lemma_path}}.
Thus, if the paths $\gamma_{q(l)}$ and $\gamma_{l}$ are both positive, $x_l \in \gamma_{l-1}$. 

In the same way, if $\gamma_{q(l)}$ and $\gamma_{l}$ are both negative, then $x_l \in \gamma_{l+1}$. In any case, for each $l  \in \{1, \dots P\}$, $x_l \in \gamma_{l-1} \cup \gamma_{l+1}$ (with the obvious convention $\gamma_{P+1}=\gamma_1$, $\gamma_{0}=\gamma_P$).

Now, suppose that for a given $k \in \{1, \dots P\}$, $x_k$ is in $\gamma_{k-1}$. Then, $\gamma_{k-1}$ and $\gamma_{k}$ have the same direction. From lemma~{\ref{lemma_path}}, it follows that $x_{k-1}$ cannot be contained in $\gamma_k$. Since we know that $x_{k-1} \in \gamma_{k-2} \cup \gamma_{k}$, $x_{k-1}$ must be in $\gamma_{k-2}$. Recursively, for each $l \in \{1, \dots P\}$, $x_l  \in \gamma_{l-1}$. It follows that for each $l \in \{1, \dots P\}$, $d(x_l, y_{\sigma(l-1)})  <  d(x_{l-1}, y_{\sigma(l-1)})$, and since $h$ is increasing
\begin{equation}
  \label{eq:2}
  \sum_{l=1}^P c(x_l, y_{\sigma(l)})  >   \sum_{l=1}^P c(x_{l+1}, y_{\sigma(l)}),
\end{equation}
which contradicts the fact that $\sigma$ is a minimizer of~(\ref{distance_MKsigma}). We come to the same conclusion if for a given $k \in \{1, \dots P\}$, $x_k$ is in $\gamma_{k+1}$
\end{proofof}

The same result can be proven for any convex function $h$ with the difference that it is only satisfied for a good choice of the permutation $\sigma$ which minimizes~(\ref{distance_MKsigma}), and not for all of these permutations. This result can be seen as a limit version of proposition~\ref{prop1}.

\begin{corollary}
\label{corollaire1}
\textit{Assume that $h$ is convex. Let $x_1,\dots x_P$ and $y_1,\dots y_P$ be $P$ points in $[0,1[$. Assume that all these points are pairwise different. 
Then there exists a permutation $\sigma$ of $\Sigma_P$ which minimizes~(\ref{distance_MKsigma}) and a point $x_k \in \{x_1, \dots x_P\}$ such that for all $l \neq k$, $x_k \notin \gamma(x_l,y_{\sigma(l)})$.}
\end{corollary}

\begin{proof}
  We know that for any strictly convex function $h$, if ${\sigma}_h$ minimizes the cost $\sigma \mapsto W_{\sigma}^c\,(f,g)$, there exists $k \in \{1, \dots P\}$ such that for all $l \neq k$, $x_k \notin \gamma_l=\gamma(x_l,y_{{\sigma}_{h}(l)})$.

Now, assume that $h$ is convex (not strictly). One can always find a sequence $(h^n)$ of increasing and strictly convex functions such that $h^n$ converges pointwise towards $h$ when $n \rightarrow \infty$. If $\sigma$ and the points $x_1,\dots x_P, y_1,\dots y_P$ are fixed, then the finite sum $ W_{\sigma}^n \,(f,g):= \frac 1 P \sum_k h^n(d(x_k,y_{\sigma(k)}))$ tends towards  $W_{\sigma} \,(f,g)= \frac 1 P \sum_k h(d(x_k,y_{\sigma(k)}))$ when $n \rightarrow \infty$. Thus, for each $\eps >0$, there exists an integer $N$, such that for all $n \geq N$, $|  W_{{\sigma}}^{n}\,(f,g) - W_{{\sigma}}\,(f,g) | \leq \eps$. Since $\Sigma_P$ is a finite set, we can chose $N$ large enough  such that this property holds for every $\sigma$ in $\Sigma_P$. We can also chose $N$ such that $| \min_{\sigma} W_{\sigma} \,(f,g) - \min_{\sigma} W_{\sigma}^{n}\,(f,g)  | \leq \eps$. Now, if $n\geq N$ and if $\sigma^*$ is an optimal permutation for $W_{{\sigma}}^{n}\,(f,g)$, it follows that
 \begin{eqnarray*}
| \min_{\sigma} W_{\sigma}\,(f,g) - W_{{\sigma^*}}\,(f,g) |  &\leq&   
 | \min_{\sigma} W_{\sigma}\,(f,g) - \min_{\sigma} W_{\sigma}^n\,(f,g)  |  \\
 && \quad +  |  W_{{\sigma^*}}^{n}\,(f,g) - W_{{\sigma^*}}\,(f,g) | \\
&\leq& 2 \eps. 
\end{eqnarray*}
Since $\Sigma_P$ is a finite set, the fact that this distance can be made arbitrarily small implies that when $n$ is large enough, a minimizer ${\sigma^*}$ of $W^{n}_{\sigma} (f, g)$ is also a minimizer of $ W _{\sigma} (f, g)$. This proves that there exists at least one minimizer $\sigma$ of $\sigma \mapsto W_{\sigma} (f, g)$ such that $x_k \notin \gamma(x_l, y_{\sigma_\lambda(l)})$ for some $k\in\{1, \dots, P\}$ and all $l\neq k$.
\end{proof}

\medskip
\subsection{Proof of Theorem~\ref{theoremMK}}

\begin{proofof}{Theorem~\ref{theoremMK}}

Let us begin with the case where $f$ and $g$ can be written as sums of unitary masses (Equation~\eqref{eq:unitary_masses}), and where $x_1,\dots x_P$ and $y_1,\dots y_P$ are pairwise different.
Proposition~{\ref{prop1}} and  Corollary~{\ref{corollaire1}} show that if the ground cost $c$ can be written $c(x,y) = h(d(x,y))$ with $h$ a positive, convex and increasing function, we can choose some optimal permutation $\sigma$ for which there is some point $x_k$ which is not contained in any  path  of $\sigma$ (recall that paths are defined as open: they do not contain their boundaries). Since all points are supposed pairwise different, the only path meeting all the neighborhoods of $x_k$ is $\gamma_k$. It follows that there exists some open set on one side of $x_k$  and not containing $x_k$ which does not cross any path of the optimal permutation $\sigma$. The middle $x$ of this open set is not contained in any path of $\sigma$. We can thus cut the circle $S^1$ at $x$ and reduce the transportation problem on the circle to the transportation problem on the real line. The optimal permutation $\sigma$ is thus given by the sorting of the points (formula~(72) in~\cite{ref:villani03book}), taking $x$ as the reference point on the circle. This means that when  points are pairwise different, we have
\begin{equation}
  \label{eq:MKresult2}
  \textsc{MK}_{c}\,(f,g)=\inf_{x \in S^1} \int_0^1 h(| F_x^{-1}-G_x^{-1} |) ,
\end{equation}
where $F_x^{-1}$ and $G_x^{-1}$ are the pseudo-inverses  (pseudo-inverses are defined in Section~\ref{sec:computing_Monge_circle}) of the increasing functions $F_x$ and $G_x$ defined in Equation~\eqref{eq:cumulative}. 

Now, observe that $F_x$ and $G_x$ are horizontal translations of $F-F(x)$ and $G-G(x)$ by the same vector $x$. In consequence,
\begin{equation}
 \int_0^1 h(| F_x^{-1}-G_x^{-1} |)  = \int_0^1 h(| (F-F(x))^{-1}-(G-G(x))^{-1} |).
\end{equation}
Since $F$ and $G$ have been defined on $\R$ such that for all $y$, $F(y+1) = F(y)+1$ and $G(y+1) = G(y)+1$, the bounds of this integral can be replaced by any bounds $(t,t+1)$. It follows that 
\begin{eqnarray*}
 \int_0^1 h(| (F-F(x))^{-1}-(G-G(x))^{-1} |) &=&  \int_{-F(x)}^{1-F(x)} h(| (F-F(x))^{-1}-(G-G(x))^{-1} |)\\
&=& \int_{0}^{1} h(| (F)^{-1}-(G+F(x)-G(x))^{-1} |).
\end{eqnarray*}
Finally,
\begin{equation}
 \textsc{MK}_{c}\,(f,g) =\inf_{x \in S1}  \int_01 h(| (F)^{-1}-(G+F(x)-G(x))^{-1} |).
\end{equation}
In order to conclude, notice that the function $\varphi : \alpha \mapsto \int_01 h(| (F)^{-1}-(G+\alpha)^{-1} |)$ is continuous ($h:\R \to \R^+$ is continuous since it is convex) and coercive ( $\varphi(\alpha) \rightarrow +\infty$ when $|\alpha| \rightarrow +\infty$). It follows that $\varphi$ reaches its minimum at a point $\alpha_0 \in \R$.  In addition, the fact that $F$ and $G$ are piecewise constant implies that $\varphi$ is piecewise affine, with discontinuities of $\varphi'$ at points $F(x)-G(x)$. Thus,
\begin{equation}
  \label{eq:MKresult_final}
\textsc{MK}_{c}\,(f,g)  =\inf_{\alpha \in \R} \int_01 h(| (F)^{-1}-(G+\alpha)^{-1} |).
\end{equation}
\medskip

The previous result can be generalized to the case where the points  $x_i$, $y_j$ may coincide just by remarking that both quantities in Equation \eqref{eq:MKresult_final} are continuous in the positions of these points. In consequence, the result holds for distributions with rational masses.

In order to generalize the result to any couple of discrete probability distributions, observe that the right term in Equation~\eqref{eq:MKresult_final} is continuous in the values of the masses $f[i]$ and $g[j]$. As for the continuity of $\textsc{MK}_c(f,g)$, assume that a mass $\eps$ of the distribution $f$ is transferred from the point $x_{i_0}$ to the point $x_{i_1}$ in $f$, and let us call the new distribution $f^{\eps}$. If $(\alpha)$  is an optimal transport plan between $f$ and $g$, let $j_0$ be an index such that $\alpha_{i_0,j_0} \geq \eps$. A transport plan $(\alpha')$ between $f^{\eps}$ and $g$ can be defined as
\begin{itemize}
\item[$\bullet$] $\alpha'_{i_0,j_0} = \alpha_{i_0,j_0} - \eps$,
\item[$\bullet$] $\alpha'_{i_1,j_0} = \alpha_{i_1,j_0} + \eps$,
\item[$\bullet$]  $\alpha'_{i,j} = \alpha_{i,j}$ for $(i,j) \neq (i_0,j_0),  (i_1,j_0)$.
\end{itemize}
The corresponding transportation cost between $f^{\eps}$ and $g$ is then lower than $\textsc{MK}_c (f,g)+\eps h(\frac 1 2)$, which implies that $\textsc{MK}_c (f^{\eps},g)\leq \textsc{MK}_c (f,g)+\eps h(\frac 1 2)$. Conversely, we can show that $\textsc{MK}_c (f,g)\leq \textsc{MK}_c (f^{\eps},g)+\eps h(\frac  1 2)$.
\end{proofof}

\end{normalsize}
\end{appendix}


\bibliographystyle{alpha}  
\bibliography{biblio_yann}

\end{document}